\documentclass[leqno,twoside,11pt]{amsart}

\usepackage{latexsym,esint,url, amssymb, yfonts}
\usepackage{graphicx,color}

\theoremstyle{plain}
\def\endproof{\hspace*{\fill}\mbox{\ \rule{.1in}{.1in}}\medskip }

\newtheorem{theorem}{Theorem}[section]
\newtheorem{corollary}[theorem]{Corollary}
\newtheorem{lemma}[theorem]{Lemma}
\newtheorem{proposition}[theorem]{Proposition}

\theoremstyle{definition}
\newtheorem{example}[theorem]{Example}
\newtheorem{remark}[theorem]{Remark}
\newcommand{\R}{\mathbb{R}} 

\newcommand{\bee}{\begin{equation}}
\newcommand{\eee}{\end{equation}} 
\newcommand{\bees}{\begin{equation*}}
\newcommand{\eees}{\end{equation*}}

\numberwithin{equation}{section}
\numberwithin{figure}{section}

\begin{document}

\title[the metric-restricted inverse design problem]
{The metric-restricted inverse design problem}
\author{Amit Acharya, Marta Lewicka and Mohammad Reza Pakzad }
\address{Amit Acharya, Civil \& Environmental Engineering,
Carnegie Mellon University, Pittsburgh, PA 15213-3890}
\address{Marta Lewicka and Mohammad Reza Pakzad, University of Pittsburgh, Department of Mathematics, 
139 University Place, Pittsburgh, PA 15260}
\email{acharyaamit@cmu.edu, lewicka@pitt.edu, pakzad@pitt.edu}

\date{today}

\begin{abstract} 
We study a class of design problems in solid mechanics, leading to a variation 
on the classical question of equi-dimensional embeddability of Riemannian manifolds. 
In this general new context, we derive a necessary and sufficient existence condition, 
given through a system of total differential equations, and
discuss its integrability. In the classical context, 
the same approach yields conditions of immersibility of a given metric in terms of the
Riemann curvature tensor. In the present situation, 
the equations do not close in a straightforward manner, and successive
differentiation of the compatibility conditions leads to a 
new algebraic description of integrability. 
We also recast the problem in a variational setting and analyze the infimum 
of the appropriate incompatibility energy, resembling the 
non-Euclidean elasticity. We then derive a $\Gamma$-convergence result
for dimension reduction from $3$d to $2$d in the Kirchhoff energy scaling regime. 
\end{abstract}

\maketitle

\section{The metric-restricted inverse design problem}

Assume ${\mathcal T}$ is a manifold 
of material types, differentiated by their structure, density,
swelling-shrinkage rates and other qualities.
We let any material type $T \in \mathcal{T}$ be naturally endowed  with a 
prestrain  $\bar g(T)$, where $\bar g: \mathcal{T} \to {\mathbb
  R}^{2\times 2}_{\rm sym,pos}$ is a given smooth mapping, taking values in
the symmetric positive definite tensors. 

Suppose now that we need to manufacture a $2$-dimensional membrane $S\subset
\R^3$, where at any given point $p\in S$ a  material of type 
$T(p)\in{ \mathcal T}$    must be used for a given $T: S\to {\mathcal T}$. The question is how to print a 
thin film $U\subset \R^2$ in a manner that the activation $u:U \to \R^3$ of the prestrain 
in the film would result in  a deformation leading eventually to the
desired surface shape $S$. 

\medskip

The above described problem is natural as a design question in various areas of solid mechanics, even though
the involved tensors are not intrinsic geometric objects. For
example, it includes the subproblems and extensions to higher dimensions:
\begin{itemize}
\item[(a)] Given the deformed configuration of an elastic $2$-dimensional membrane and the
  rectangular Cartesian components of the Right Cauchy-Green tensor
  field of a deformation, mapping a flat undeformed reference of the
  membrane to it, find the flat reference configuration and the
  deformation of the membrane\footnote{
We thank Kaushik Bhattacharya for bringing this problem to our attention.}.
\item[(b)] Given the deformed configuration of a $3$-dimensional body and the
  rectangular Cartesian components of the Right Cauchy-Green tensor
  field of the deformation, mapping a reference configuration to it,
  find the reference configuration and the deformation. 
\item[(c)] Suppose the current configuration of a $3$-dimensional, plastically deformed
  body is given, and on it is specified the rectangular Cartesian
  components of a plastic distortion $F_p$. Find a reference
  configuration and a deformation $\zeta$, mapping this reference to the given
  current configuration, such that the latter is stress-free. Assume
  that the stress response of the material is such that the stress
  vanishes if and only if $(\nabla\zeta (F_{p})^{-1})^T (\nabla\zeta (F_{p})^{-1}) = \mbox{Id}_3$.
\end{itemize}

\medskip

In view of \cite{klein, Sharon2, lepa}, the activation 
$u$ must be an isometric immersion of the Riemannian manifold $(U, \mathcal{G})$ into $\R^3$, 
where $\mathcal {G}$ is the prestrain in the flat (referential) thin film.
Our design problem requires hence that we find {\em an unknown reference configuration} $U\subset \R^2$,  
{\em an unknown material distribution}
${\mathfrak T}: U \to \mathcal{T}$ and {\em an unknown deformation} $u: U \to \R^3$ such that 
\begin{itemize}
\item[(i)] $S= u(U)$, 
\item[(ii)] For any $x \in U$, the point  $u(x)$ 
carries a material of type $T(u(x))$, i.e. $T(u(x)) = \mathfrak{T}(x)$,
\item[(iii)] $\nabla u(x)^T \nabla u(x) = {\mathcal G} (x) := \bar g(\mathfrak{T}(x)).$
\end{itemize}

\smallskip

\noindent If the membrane $S\subset\mathbb{R}^3$ is  
a smooth surface, then letting $g:= \bar g\circ T: S\to \R^{2\times2}_{sym, pos}$,
the conditions (i)-(iii) simplify to finding a domain $U\subset\mathbb{R}^2$ and a
bijection $u:U\rightarrow S$, such that:
\begin{equation}\label{jeden}
(\nabla u)^T(\nabla u) (x) = g(u(x)) \qquad \forall x\in U.
\end{equation}
The smoothness of $g$ is determined by the regularity of $\mathcal{T}$
and of the mappings $\bar g$ and $T$.

\medskip

Essentially, in all of the applications defined above (e.g. membrane,
3-d), we are dealing with a general class of nonlinear elastic
constitutive assumptions involving pre-strain, with the requirement
that  the stored energy density evaluated at the Identity tensor (of
appropriate dimensionality) attain the value zero; in other words, we
look for stress-free deformations of a prestrained body. Given a
prestrain field \emph{specified on the target configuration}, we
explore the question of existence of deformations that allow such a
minimum energy state to be attained pointwise, as well as the
characterization of the constraints on the pre-strain field that
allows such attainment. In the language of mechanics, note that the
question (\ref{jeden}) may be rephrased as looking for deformations
$u$ of the reference $U$ such that:
\[
\left[ \nabla u \left(\sqrt{ g(u)}^{-1}\right) \right]^T \left[ \nabla
  u \left(\sqrt{ g(u)}^{-1}\right) \right] = \mbox{Id}, 
\] 
where the expression on the left-hand-side of the above equality is
the sole argument of the frame-indifferent nonlinear elastic energy
density function of the material. Thus, $\sqrt{g(u)}^{-1}$ needs to be capable of annihilation by the
right stretch tensor of a deformation, a condition that is expressed
in terms of spatial derivatives of $g(u)$ on $U$; the main difficulty
is that both $u$ and $U$ are unknown, so this differential condition
cannot simply be written down and a more sophisticated idea than the
standard vanishing of the Riemann-Christoffel tensor of a metric is
needed. 

\medskip

The inverse design problem that we study in this paper, can be further rephrased as follows. 
Let $y:\Omega\rightarrow \mathbb{R}^3$, be a smooth parametrization of
$S=y(\Omega)$. Find a change of variable
$\xi:\Omega\rightarrow U$ so that the pull back of the
Euclidean metric on $S$ through $y$ is realized by the following formula: 
\begin{equation}\label{dwa}
 (\nabla y)^T\nabla y   = (\nabla\xi)^T (g\circ y) \nabla \xi  \qquad
\mbox{in } \Omega.
\end{equation}
Clearly, once a solution $\xi$ of \eqref{dwa} is found, 
the  material type distribution $\mathfrak{T}$, which is needed for the construction of the printed film $U$, 
can be calculated by:
$$ \mathfrak T:=   T\circ y\circ \xi^{-1} :U=  \xi(\Omega) \to\mathcal{T}, $$ 
since $u=y\circ\xi^{-1}$ satisfies \eqref{jeden} and consequently
the properties (i)-(iii) hold.
Any $\xi$ satisfying (\ref{dwa}) is an isometry between the Riemannian manifolds $(\Omega,
\tilde G)$  and $(U, G\circ \xi^{-1})$, with  metrics:
\begin{equation}\label{trzy}
\tilde G=(\nabla y)^T\nabla y \quad \mbox{ and } \quad G=g\circ y
\quad \mbox{ on }~~~\Omega.
\end{equation}
For convenience of the reader, we gather some of our notational symbols in
Figure \ref{fig:geometry}. 

\medskip
 
The same problem can be set up for a three dimensional shape $S
\subset \R^3$. In that case, the prestrain mapping $\bar g$ must take values in $\R^{3\times 3}_{sym, pos}$ and
the printed prestrained reference configuration is modeled by  
$U\subset \R^3$. The equation to be solved is still given by
\eqref{dwa}, now posed in $\Omega \subset \R^3$. Equivalently, a solution $u:U \to S$ to
\eqref{jeden} is obtained as in (iii) above for $\mathfrak T: = T
\circ u$ and it is the absolute minimizer of the prestrain elastic energy (see e.g. \cite{lepa}): 
\begin{equation}\label{uenergy}
 E(u, U, {\mathcal G}): = \int_U {\rm dist^2}\big(  (\nabla u) {\mathcal G} ^{-1/2}, SO(3)\big)~\mbox{d}x =0.
\end{equation} 
Here,  ${\rm dist} (F, SO(3))$ stands for the distance of a matrix $F$ from the compact set 
$SO(3)$, with respect to the Hilbert-Schmidt norm. 
What distinguishes our problems from the classical isometric immersion problem in differential
geometry, where one looks for an isometric mapping between two given
manifolds $(\Omega, \tilde G)$ and $(U, \mathcal{G})$, 
is that the target  manifold $U=\xi(\Omega)$  
and its Riemannian metric $\mathcal{G}= G\circ \xi^{-1}$   are  only
given a-posteriori, after the solution is found.  Note that only when $G$ 
is constant, the target metric becomes a-priori well defined and can be extended over
the whole of $\R^n$, as it is independent of $\xi$, and then the
problem reduces to the classical case (see Example \ref{ex4.4} and a
few other similar cases in Examples \ref{ex4.5}  and \ref{lambda-mu}).
\begin{figure}[h]
\vspace{-5mm}
\hspace{-0.6cm}\includegraphics[width=5.25in, height=3.5in]{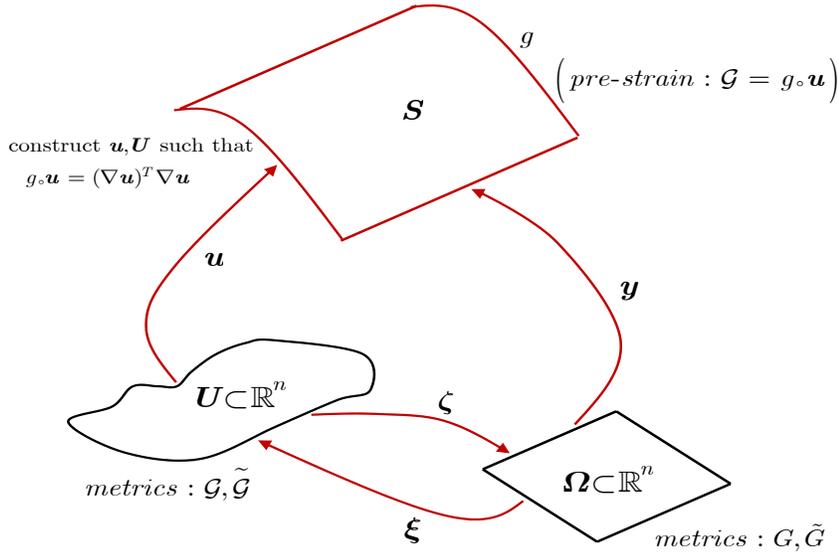}
\vspace{-0.7cm} \caption{Geometry of the problem.}
\label{fig:geometry}
\end{figure}

\begin{remark}
Note that the minimizing solution in the equidimensional problem (\ref{uenergy}) is unique up to rigid motions.
Hence, If  the configuration $(U,{\mathcal G})$ is printed, the elastic  
body is bound to take the required shape $S$ as requested in the
design problem. Note that in case the required shape $S$   
is two dimensional, uniqueness fails due to one more degree of freedom as the deformation 
$u$ of $U\subset \R^2$ takes values in a higher dimensional space
$\R^3$. In this case, other restrictions on $u$ have to be imposed, to
solve the original design problem. One particular remedy is to
consider the membrane as a thin three dimensional body; we implement
this approach in Section \ref{dimred}.  
Note that solving \eqref{jeden} 
directly in two dimensions implies that a compatible reference
configuration can be found, as formulated in (a) above.    

Another approach in the two dimensional case would be to use 
a parametrization $y:\Omega \to S$ which is bending energy minimizing
among all  other mappings $\tilde y:\Omega \to S$  that induce the
same metric $(\nabla y)^T \nabla y$. This would necessitate a study of the bending energy effects, which is beyond 
the scope of the present paper. For the sake of comparison, we mention
a parallel problem, studied in \cite{kim}, where the authors calculate
the prestrain to be printed in the thin film by minimizing a full
three dimensional energy of deformation, consisting of both stretching
and bending. The nature of the problem in \cite{kim} is different than
ours, in as much as the material constraints are not present.
\end{remark}

\medskip

The paper is organized as follows.
In Sections 2 and 3, we study the above mentioned 
variational formulation of (\ref{dwa}) and analyze the infimum value
of the appropriate incompatibility energy, resembling the
non-Euclidean elasticity \cite{lepa}. We derive a
$\Gamma$-convergence result for the dimension reduction from $3$d to
$2$d, in the Kirchhoff-like energy scaling regime, corresponding to the
square of thickness of the thin film. In Sections 4-9, we formulate
(\ref{dwa}) as an  algebraically constrained 
system of total differential equations, in which the second
derivatives of $\xi$ are expressed in terms of its first derivatives
and the Christoffel symbols of the involved metrics.  The idea is then 
to investigate the integrability conditions of this
system. When this method is applied in the context of the standard 
Riemannian isometric immersion 
problem, the parameters involving $\xi$ can be  removed from the
conditions and the intrinsic conditions of immersibility will be given
in terms of the Riemann curvature tensors. In our situation, the
equations do not close in a straightforward manner, and successive
differentiation of the compatibility conditions leads to a more
sophisticated algebraic description of solvability. This approach has 
been adapted in \cite{LCG-acha} for deriving 
compatibility conditions for the Left Cauchy-Green tensor. 


\bigskip

\noindent {\bf Acknowledgments.}
This project is based upon work supported by, among others, the National Science
Foundation. A.A. acknowledges support in part from grants NSF-CMMI-1435624, NSF-DMS-1434734, and ARO W911NF-15-1-0239. M.L. was partially supported by the NSF grants DMS-0846996 and DMS-1406730. M.R.P. was partially supported by the NSF grant
DMS-1210258. A part of this work was completed while the second and
the third authors visited the Forschungsinstitut f\"ur Mathematik at ETH 
(Zurich, Switzerland). The institute's hospitality is gratefully acknowledged.

\section{A variational reformulation of the problem (\ref{trzy})}\label{var_reform}

In this section, we recast the problem \eqref{trzy} in a variational
setting, similar to that of non-Euclidean elasticity \cite{lepa}. Using
the same arguments as in \cite{FJMgeo, lepa}, we will analyze the
properties of the infimum value of the appropriate incompatibility energy,
over the natural class of deformations of $W^{1,2}$ regularity.

We will first discuss the problem \eqref{dwa} in the general $n$-dimensional
setting and only later restrict to the case $n=2$ (or $n=3$). 
Hence, we assume that $\Omega$ is an open, bounded, simply connected and
smooth subset of $\mathbb{R}^n$. We look for a bilipschitz map $\xi:\Omega\to
U:=\xi(\Omega)$, satisfying (\ref{trzy}) and which is orientation preserving:
\begin{equation}\label{orienta}
\det\nabla\xi > 0 \quad \mbox{ in } ~\Omega. 
\end{equation} 

\medskip

We begin by rewriting (\ref{trzy}) as: 
\begin{equation}\label{trzy3}
\tilde G = \big(G^{1/2}\nabla \xi\big)^T \big(G^{1/2}\nabla
\xi\big).
\end{equation}
Note that, in view of the polar decomposition theorem of
matrices,  a vector field $\xi:\Omega\to\mathbb{R}^n$
is a solution to (\ref{trzy3}), augmented by the constraint
(\ref{orienta}), both valid a.e. in $\Omega$, if and only if:
\begin{equation}\label{equ-rotation} 
\forall a.e.~ x\in \Omega \quad \exists  R= R(x)\in SO(n) \qquad G^{1/2}
\nabla \xi   = R ~\tilde G^{1/2},
\end{equation} 
where $G^{1/2}(x)$ denotes the unique symmetric positive definite
square root of $G(x)\in\mathbb{R}^{n\times n}_{sym, pos}$, while
$SO(n)$ stands for the set of special orthogonal matrices. Define:
\begin{equation}\label{energy}
E(\xi) = \int_\Omega {\rm dist^2}\big( G^{1/2} (\nabla\xi) \tilde
G^{-1/2}, SO(n)\big)~\mbox{d}x \qquad \forall \xi \in W^{1,1}_{loc}(\Omega, \R^n), 
\end{equation}
where  ${\rm dist} (F, SO(n))$ is the calculated distance of a matrix $F$ from the compact set 
$SO(n)$, with respect to the Hilbert-Schmidt norm of matrices.
It immediately follows that $E(\xi)=0$  if and only if  $\xi$ is a solution to
\eqref{trzy3} and hence to (\ref{trzy}), together with
(\ref{orienta}). Also, note that $E(\xi)<\infty$ if and only if $\xi \in
W^{1,2}(\Omega, \R^n)$, as can be easily deduced from the inequality: 
\begin{equation}\label{coerce}
\forall F\in \R^{n\times n} \qquad |F|^2 \le C |G^{1/2} F \tilde G^{-1/2}|^2,
\end{equation} 
valid with a constant $C>0$ independent of $x$ and $F$.

Finally, observe that, due to the uniform positive definiteness of the matrix field $G$:
\begin{equation}\label{length} 
 |F |^2= \sum_{i=1}^n  |Fe_i|^2  \le C \sum_{i=1}^n \langle Fe_i,  G
 Fe_i\rangle \leq C\mbox{trace }(F^TG F) . 
\end{equation}  

\begin{proposition}\label{regularity}
(i) Assume that the metrics $G, \tilde G $ are $\mathcal{C}(\bar\Omega,
\R^{n\times n})$ regular. Let  $\xi\in W^{1,1}_{loc} (\Omega,\mathbb{R}^n)$  satisfy  
\eqref{trzy} for a.e. $x\in\Omega$. Then $\xi \in W^{1,\infty}(\Omega,
\R^n)$ must be Lipschitz continuous. 

(ii) Assume additionally that for some $k\ge 0$ and $0<\mu<1$,   
$G, \tilde G \in \mathcal{C}^{k,\mu}(\Omega, \mathbb{R}^{n\times
  n})$.  If  (\ref{trzy}), (\ref{orienta}) hold a.e. in $\Omega$ 
(so that $E(\xi)=0$), then $\xi \in \mathcal{C}^{k+1, \mu} (\Omega, \R^n)$.
\end{proposition}

\begin{proof}
The first assertion clearly follows from the boundedness of $\tilde G$ and
positive definiteness of $G$, through \eqref{length}.  
To prove (ii), recall that for a matrix $F\in\mathbb{R}^{n\times n}$,
the matrix of cofactors of $F$ is $\mbox{cof } F$, with 
$(\mbox{cof } F)_{ij} = (-1)^{i+j} \det \hat F_{ij}$, where 
$\hat F_{ij}\in\mathbb{R}^{(n-1)\times (n-1)}$ is obtained from $F$ by deleting 
its $i$th row and $j$th column. Then,  \eqref{trzy} implies that: 
\begin{equation*} 
\det \nabla \xi = \Big  (\frac{\det \tilde G}{\det G} \Big )^{1/2} =:
a \in \mathcal{C}(\bar \Omega, \R^+) \quad \mbox{ and } \quad
\mbox{cof} \nabla \xi = a   G (\nabla \xi) \tilde G^{-1}.
\end{equation*}
Since $\mbox{div}(\mbox{cof } \nabla \xi) = 0$ for $\xi\in W^{1,\infty}$ (where the divergence 
of the cofactor matrix is always taken row-wise), we obtain that $\xi$
satisfies the following linear system of differential equations, in the weak sense: 
\begin{equation*}
{\rm div} \big ( a G (\nabla \xi) \tilde G^{-1}\big ) =0.
\end{equation*} 
Writing in coordinates $\xi=(\xi^1\ldots \xi^n)$, and using the
Einstein summation convention, the above system reads:
\begin{equation*}
\forall i=1\ldots n \qquad \partial_\alpha  \big ( a G_{ij} \tilde
G^{\beta\alpha}  \partial_\beta \xi^j  \big ) =0.
\end{equation*} 
The regularity result is now an immediate consequence of
\cite[Theorem 3.3]{gia} in view of the ellipticity of the coefficient matrix $A^{\alpha\beta}_{ij}= a G_{ij} \tilde
G^{\alpha\beta}$.  
 \end{proof}

We now prove two further auxiliary results.

\begin{lemma}\label{lem2.3}
There exist constants $C, M>0$, depending only on $\|G\|_{L^\infty}$ and $\|\tilde G\|_{L^\infty}$, such that
for every $\xi \in W^{1,2}(\Omega,\mathbb{R}^n)$ there exists $\bar
\xi\in W^{1,2}(\Omega,\mathbb{R}^n)$ with the properties: 
$$ \|\nabla\bar \xi \|_{L^\infty} \leq M, \qquad  
\|\nabla \xi - \nabla\bar \xi \|_{L^2(\Omega)}^2\leq C E(\xi)
\qquad \mathrm{and} \qquad E(\bar \xi)\leq C E(\xi).$$
\end{lemma}
\begin{proof}
Use the approximation  result of Proposition A.1. in \cite{FJMgeo} to
obtain the truncation $\bar \xi = \xi^\lambda$, for $\lambda>0$ 
having the property that if a matrix $F\in \mathbb{R}^{n\times n}$
satisfies $|F|\geq\lambda$ then:
\begin{equation*} 
|F|^2\leq C \mbox{dist}^2( G ^{1/2} F\tilde G^{-1/2} (x), SO(n)) \qquad \forall x\in \Omega.
\end{equation*} 
Then $\|\nabla \xi^\lambda\|_{L^\infty}\leq C\lambda:= M$ and further,
since $\nabla\xi = \nabla\xi^\lambda$ a.e. in the set $\{|\nabla\xi|\leq\lambda \}$:
$$\|\nabla \xi - \nabla \xi^\lambda\|_{L^2(\Omega)}^2 = \int_{\{|\nabla \xi|>\lambda\}}
|\nabla \xi|^2 \leq c \int_{\{|\nabla \xi|>\lambda\}} 
\mbox{dist}^2( G ^{1/2} \nabla \xi \tilde G^{-1/2}, SO(n))~\mbox{d}x \leq C E(\xi).$$
The last inequality of the lemma follows from the above by the triangle inequality.
\end{proof}

\begin{lemma}\label{lem2.4}
Let $\xi \in W^{1,\infty}(\Omega,\mathbb{R}^n)$. Then there exists a
unique weak solution $\phi:  \Omega\to \R^n$ to:    
\begin{equation}\label{EL-multiple}  
\left\{\begin{split}
&  {\rm div} \big ( a G (\nabla \phi) \tilde G^{-1}\big )
=0 \quad \mbox{ in } \Omega, \\
& \phi = \xi \quad \mbox{ on } \partial\Omega.
\end{split}\right.
\end{equation} 
Moreover, there is constant $C>0$,  depending only on $G$ and $\tilde G$, and  (in a nondecreasing
manner) on $\|\nabla \xi\|_{L^\infty}$, such that:
$$\|\nabla (\xi-\phi)\|_{L^2(\Omega)}^2\leq C E(\xi).$$
\end{lemma}
\begin{proof}
Consider the functional:
\begin{equation*}
I(\varphi):= \int_\Omega \langle G(\nabla\varphi)\tilde G^{-1}(x) :
\nabla\varphi(x)\rangle~\mbox{d}x = \int_\Omega |a^{1/2} G^{1/2} (\nabla \varphi) \tilde
G^{-1/2}|^2~\mbox{d}x \qquad \forall \varphi \in W^{1,2}(\Omega, \R^n).    
\end{equation*}
The formula \eqref{coerce}, in which we have implicitly used the  coercivity of $G$ and $\tilde G$,  implies that:  
\begin{equation*} 
\|\nabla \varphi\|^2_{L^2(\Omega)} \le C I(\varphi).
\end{equation*} 
Therefore, in view of the strict convexity of $I$, the direct method
of calculus of variations implies that  $I$ admits a  unique critical
point $\phi$ in the set: 
\begin{equation*}
\big\{\varphi \in W^{1,2}(\Omega, \R^n); \,\, \varphi = \xi \,\, \mbox{on} \,\, \partial \Omega \big\}. 
\end{equation*} 
By the symmetry of $G$ and $\tilde G$, \eqref{EL-multiple} is
precisely the Euler-Lagrange equation of $I$, 
and hence it is satisfied, in the weak sense, by $\phi$. 
 
Further, for the correction $\psi=\xi-\phi \in W^{1,2}_0(\Omega,\mathbb{R}^n)$ it follows that:
\begin{equation*}
\begin{split} 
\forall \eta\in W^{1,2}_0(\Omega, \mathbb{R}^n) \qquad
\int_\Omega a G (\nabla \psi)  \tilde G^{-1} : \nabla \eta ~\mbox{d}x 
& =  \int_\Omega a G (\nabla \xi)  \tilde G^{-1} : \nabla \eta
- \int_\Omega a G (\nabla \phi)  \tilde G^{-1} : \nabla \eta \\   
& = \int_\Omega a G (\nabla \xi)  \tilde G^{-1} : \nabla \eta \\ & 
=  \int_\Omega a G (\nabla \xi)  \tilde G^{-1} : \nabla \eta
- \int_\Omega  {\rm cof} \nabla \xi : \nabla \eta.
\end{split} 
\end{equation*} 
Indeed, the last term above equals to $0$, since the row-wise divergence of the cofactor
matrix of $\nabla \xi$ is $0$, in view of $\xi$ being Lipschitz continuous.
Use now $\eta=\psi$ to obtain:
\begin{equation*}\label{glupie3}
\begin{split}
\|\nabla \psi\|^2_{L^2(\Omega)}  & \le  C I(\psi)  = 
 C \int_\Omega (a G (\nabla \xi)  \tilde G^{-1}  - {\rm cof} \nabla \xi) : \nabla \psi  ~\mbox{d}x 
\\ & \le C \|\nabla \psi \|_{L^2(\Omega)} \left(\int_{\Omega}\left | a
    G (\nabla \xi)  \tilde G^{-1}  - {\rm cof} \nabla \xi \right
  |^2\right)^{1/2} \\ &\leq C\|\nabla \psi \|_{L^2(\Omega)}  E(\xi)^{1/2}.
\end{split}
\end{equation*}
The last inequality above follows from:
\begin{equation*}
\forall |F|\leq M \quad \forall x\in\Omega \qquad  |aGF\tilde G^{-1}
(x) - {\rm cof} \,F|^2 \le C_M {\rm dist}^2 \big(G^{1/2}F\tilde G^{-1/2}, SO(n)\big),   
\end{equation*} 
because when $G^{1/2}F\tilde G^{-1/2}\in SO(n)$ then the difference in
the left hand side above equals $0$.  
\end{proof}

\begin{theorem}\label{energy-condition}
Assume that the metrics $G, \tilde G \in \mathcal{C}(\bar \Omega,
\R^{n\times n})$ are Lipschitz continuous. Define: 
\begin{equation}\label{defect}
\kappa(G, \tilde G)= \inf_{\xi \in W^{1,2}(\Omega, \R^n)} E(\xi).  
\end{equation} 
Then,  $\kappa(G, \tilde G)=0$ if and only if there exists a
minimizer $\xi \in W^{1,2}(\Omega, \R^n)$ with $E(\xi)=0$.
In particular,  in view of Proposition \ref{regularity}, this is equivalent to $\xi$ being a solution to
(\ref{trzy}) (\ref{orienta}), and  $\xi$ is smooth if  $G$ and $\tilde G$ are smooth.   
\end{theorem}
\begin{proof} 
Assume, by contradiction, that for some sequence of deformations
$\xi_k\in W^{1,2}(\Omega, \mathbb{R}^n)$, there holds $\lim_{k\to\infty} E(\xi_k) = 0$.
By Lemma \ref{lem2.3}, replacing $\xi_k$ by $\bar \xi_k$, 
we may without loss of generality request that $\|\nabla \xi_k\|_{L^\infty} \leq M$.

The uniform boundedness of $\nabla \xi_k$ implies, 
via the Poincar\'{e} inequality, and after a 
modification by a constant and passing to a subsequence, if necessary:
\begin{equation}\label{glupie2.5}
\lim _{k\to\infty}\xi_k = \xi \qquad \mbox{ weakly in } W^{1,2} (\Omega).
\end{equation}
Consider the decomposition $\xi_k=\phi_k+ \psi_k$, where $\phi_k$
solves (\ref{EL-multiple}) with the boundary data $\phi_k = \xi_k$ on $\partial\Omega$.
By the Poincar\'e inequality, Lemma \ref{lem2.4} implies 
for the sequence $\psi_k\in W_0^{1,2}(\Omega)$:
\begin{equation*}
\lim_{k\to\infty} \psi_k = 0 \qquad \mbox{ strongly in } W^{1,2}(\Omega).
\end{equation*}
In view of the convergence in (\ref{glupie2.5}),  the sequence $\phi_k$ must
be uniformly bounded in $W^{1,2}(\Omega)$, and hence by \cite[Theorem 4.11, estimate (4.18)]{yan}:
\begin{equation*} 
\forall \Omega '\subset\subset  \Omega \quad \exists C_{\Omega'}
\quad \forall k \qquad \|\phi_k\|_{W^{2,2}(\Omega')} \leq C_{\Omega'} 
\|\phi_k\|_{W^{1,2}(\Omega)}\leq C.
\end{equation*} 
Consequently, $\phi_k$ converge to $\phi$ strongly in $W_{loc}^{1,2}(\Omega)$.
Recalling that $E(\xi_k)$ converge to $0$, we finally conclude that:
$$ E(\xi) = 0. $$
This proves the claimed result.
\end{proof}

\section{A dimension reduction result for the energies (\ref{energy})}\label{dimred}

The variational problem induced by (\ref{energy})  is difficult due to the lack of convexity.  One
way of reducing the complexity of this problem is to assume that the
target shapes are thin bodies, described by a ``thin
limit'' residual theory, which is potentially easier to analyze. 
We concentrate on the case $n=3$ and the energy functional 
(\ref{energy}) relative to a family of thin films $\Omega^h =
\omega\times (-\frac{h}{2}, \frac{h}{2})$ with the midplate
$\omega\subset\mathbb{R}^2$ given by an open, smooth and bounded set.
This set-up corresponds to a scenario where a target $3$-dimensional thin
shell is to be manufactured, rather than a $2$-dimensional surface as
in Figure \ref{fig:geometry}. 
As we shall see below, an approximate realization of the ideal thin
shell is delivered by solving the variational problem in the limit of the
vanishing thickness $h$. 

Assume further that the given smooth metrics $G, \tilde G:\bar\Omega^h\to
\mathbb{R}^{3\times 3}$ are thickness-independent:
$$G,\tilde G (x', x_3) = G,\tilde G(x') \qquad \forall x=(x', x_3)\in\omega\times (-\frac{h}{2}, \frac{h}{2}),$$
and denote:
\begin{equation}\label{energyh}
E^h(\xi^h) = \frac{1}{h}\int_{\Omega^h} W\big(G^{1/2}(\nabla\xi^h)\tilde
G^{-1/2}\big)~\mbox{d}x \qquad \forall \xi^h\in W^{1,2}(\Omega^h,\mathbb{R}^3).
\end{equation} 
The energy density $W:\mathbb{R}^{3\times
  3}\to{\bar{\mathbb{R}}}_+$ is assumed to be $\mathcal{C}^2$ regular close to
$SO(3)$, and to satisfy the conditions of normalisation, frame
invariance and bound from below:
\begin{equation*}
\begin{split}
\exists c>0 \quad \forall F\in\mathbb{R}^{3\times 3}\quad \forall R\in SO(3)
\qquad & W(R) = 0, \quad W(RF) = W(F), \\ & W(F)\geq c~\mbox{dist}^2(F, SO(3)).
\end{split}
\end{equation*}

Following the approach of \cite{FJMgeo}, which has been further developed in
\cite{FJMhier, lepa, BLS, Raoult} (see also \cite{Raoult} for an extensive review of the literature),  
we obtain  the following $\Gamma$-convergence results, describing in a
rigorous manner the asymptotic behavior of the 
approximate minimizers of the energy (\ref{energyh}).

\begin{theorem}\label{thm2}
For a given sequence of deformations $\xi^h\in W^{1,2}(\Omega^h,\mathbb{R}^3)$ satisfying:
\begin{equation}\label{en_bound} 
\exists C>0 \quad \forall h \qquad E^h(\xi^h) \leq Ch^2,
\end{equation}
there exists a sequence of vectors $c^h\in\mathbb{R}^3$, such that the following properties hold for the
normalised deformations $y^h\in W^{1,2}(\Omega^1,\mathbb{R}^3)$:
$$y^h(x', x_3) = \xi^h(x', hx_3) - c^h.$$
\begin{itemize}
\item[(i)] There exists $y\in W^{2,2}(\omega,\mathbb{R}^3)$ such that,
  up to a subsequence:
\begin{equation*}
y^h\to y \qquad \mbox{ strongly in } W^{1,2}(\Omega^1,\mathbb{R}^3).
\end{equation*}
The deformation $y$ realizes the compatibility of metrics $G$ and
$\tilde G$ on the midplate $\omega$: 
\begin{equation}\label{cinque}
(\nabla y)^TG~\nabla y = \tilde G_{2\times 2}.
\end{equation}
The unit normal $\vec N$ to the surface $y(\omega)$ and
the metric $G$-induced normal $\vec M$ below have the regularity $ \vec
N, \vec M\in W^{1,2}\cap L^\infty (\omega,\mathbb{R}^3)$:
\begin{equation*}
{ \vec N = \frac{\partial_1 y \times\partial_2y }{|\partial_1 y
    \times\partial_2y|}} \qquad  \vec M = \frac{\sqrt{\det
    G}}{\sqrt{\det \tilde G_{2\times 2}}} ~ G^{-1} (\partial_1y\times \partial_2 y),
\end{equation*}
where we observe that $\langle \partial_iy, G\vec M\rangle = 0$ for
$i=1,2$ and $\langle \vec M, G\vec M\rangle = 1$.

\item[(ii)] Up to a subsequence, we have the convergence:
\begin{equation*}
\frac{1}{h}\partial_3y^h\to \vec b \qquad \mbox{ strongly in } L^{2}(\Omega^1,\mathbb{R}^3),
\end{equation*}
where the Cosserat vector $\vec b \in W^{1,2}\cap L^\infty (\omega,\mathbb{R}^3)$ is given by:
\begin{equation}\label{b}
\vec b = (\nabla y) (\tilde G_{2\times
  2})^{-1}\left[\begin{array}{c}\tilde G_{13}\\\tilde G_{23}\end{array}\right]
+ \frac{\sqrt{\det \tilde G}}{\sqrt{\det \tilde G_{2\times 2}}}  \vec M.
\end{equation}

\item[(iii)] Define the quadratic forms:
\begin{equation*}
\begin{split}
 \mathcal{Q}_3(F) = D^2 W(\mathrm{Id})(F,F),\\
 \mathcal{Q}_2(x', F_{2\times 2}) = \min\left\{
   \mathcal{Q}_3\Big(\tilde G(x)^{-1/2}\tilde F \tilde G(x')^{-1/2}\Big);
  ~ \tilde F\in\mathbb{R}^{3\times 3} \mbox{ with }\tilde F_{2\times 2} = F_{2\times 2}\right\}.
\end{split}
\end{equation*}
The form $\mathcal{Q}_3$ is defined for  all $F\in\mathbb{R}^{3\times 3}$, 
while $\mathcal{Q}_2(x', \cdot)$ are defined on $F_{2\times
  2}\in\mathbb{R}^{2\times 2}$. Both forms $\mathcal{Q}_3$ and all $\mathcal{Q}_2$ are nonnegative
definite and depend only on the  symmetric parts of their arguments.
In particular, when the energy density $W$ is isotropic, i.e.:
$$\forall F\in\mathbb{R}^{3\times 3}\quad\forall R\in SO(3)\qquad W(RF) = W(F),$$
then $\mathcal{Q}_2(x',\cdot)$ is given in terms of the Lam\'e
coefficients $\lambda, \mu >0$ by:
\begin{equation*}
\mathcal{Q}_2(x', F_{2\times 2}) = \mu\left| (\tilde G_{2\times
    2})^{-1/2} F_{2\times 2} (\tilde G_{2\times 2})^{-1/2}\right|^2 + \frac{\lambda\mu}{\lambda+\mu}
\left|\mathrm{tr}\left((\tilde G_{2\times 2})^{-1/2}F_{2\times
    2} (\tilde  G_{2\times 2})^{-1/2}\right)\right|^2,
\end{equation*}
for all $F_{2\times 2}\in\mathbb{R}^{2\times 2}_{sym}$.

\item[(iii)] We have the lower bound:
\begin{equation}\label{IG}
\liminf_{h\to 0} \frac{1}{h^2}E^h(\xi^h) \geq \mathcal{I}_{G,\tilde
  G}(y) := \frac{1}{24}\int_\Omega \mathcal{Q}_2
\left(x', (\nabla y)^T G ~\nabla  \vec b \right)~\mathrm{d}x'.
\end{equation}
\end{itemize}
\end{theorem}
\begin{proof} 
The convergences in (i) and (ii) rely on a version of an approximation result from \cite{FJMgeo};
there exists matrix fields $Q^h\in W^{1,2}(\omega, \mathbb{R}^{3\times
  3})$ and  a constant $C$ uniform in $h$, i.e. depending only on the
geometry of $\omega$ and on $G, \tilde G$, such that:
\begin{equation*}
\begin{split}
& \displaystyle{\frac{1}{h} \int_{\Omega^h}|\nabla \xi^h(x', x_3) -
  Q^h(x')|^2 ~\mathrm{d}x \leq C \left( h^2 +
    \frac{1}{h}\int_{\Omega^h}\mathrm{dist}^2\big(G^{1/2}\nabla
    \xi^h \tilde G^{-1/2}, SO(3)\big)~\mathrm{d}x\right)}, \\ 
& \displaystyle{\int_{\Omega}|\nabla  Q^h(x')|^2~\mathrm{d}x' \leq C 
\left(1 + \frac{1}{h^3}\int_{\Omega^h}\mathrm{dist}^2\big(G^{1/2}\nabla
  \xi^h\tilde G^{-1/2}, SO(3)\big)~\mathrm{d}x\right)}. 
\end{split}
\end{equation*}
Further ingredients of the proof follow exactly as in \cite{BLS}, so we suppress the details. 
\end{proof}

\begin{theorem}\label{thm3}
For every compatible immersion $y\in W^{2,2}(\Omega,\mathbb{R}^3)$  satisfying (\ref{cinque}),  there exists a
sequence of recovery deformations $\xi^h\in W^{1,2}(\Omega^h,\mathbb{R}^3)$, such that:
\begin{itemize}
\item[(i)]  The rescaled sequence $y^h(x', x_3) = \xi^h(x', hx_3)$ converges in
$W^{1,2}(\Omega^1,\mathbb{R}^3)$ to $y$.
\item[(ii)] One has: $$\lim_{h\to 0} \frac{1}{h^2}E^h(\xi^h) =
  \mathcal{I}_{G, \tilde G}(y),$$
where the Cosserat vector $\vec b$ in the definition (\ref{IG}) of $\mathcal{I}_{G,\tilde G}$ is derived by (\ref{b}).
\end{itemize}
\end{theorem}
\begin{proof}
Let $y\in W^{2,2}(\Omega, \mathbb{R}^3)$ satisfy
(\ref{cinque}). Define $ \vec b$ according to (\ref{b}) and let: 
$$Q= \left[\begin{array}{ccc} \partial_1y &\partial_2 y & 
    \vec b\end{array}\right] \in W^{1,2}\cap L^\infty(\omega, \mathbb{R}^{3\times 3}).$$
By Theorem \ref{thm2}, it follows that:
\begin{equation*}
G^{1/2}Q\tilde G^{-1/2}\in SO(3) \qquad \forall \mbox{a.e.}~ x'\in\omega.
\end{equation*}
Define the limiting warping field $ \vec d\in L^2(\Omega, \mathbb{R}^3)$:
\begin{equation*}
 \vec d(x') = G^{-1} Q^{T, -1}\left( c\big(x', (\nabla y)^T G ~\nabla
   \vec b\big) - \left[\begin{array}{c} \langle \partial_1\vec b, G\vec b\rangle \\ 
\langle \partial_2\vec b, G\vec b\rangle \\ 0
\end{array}\right]\right),
\end{equation*}
where $c(x', F_{2\times 2})$ denotes the unique minimizer of the problem in:
\begin{equation*}
\begin{split}
\forall F_{2\times 2}\in\mathbb{R}^{2\times 2}_{sym}\qquad 
\mathcal{Q}_2(x', F_{2\times 2}) & = \min\left\{
  \mathcal{Q}_3\big(\tilde G^{-1/2}(F^*_{2\times 2} +
  \mbox{sym}(c\otimes e_3)) \tilde G^{-1/2}\big); ~ c\in\mathbb{R}^3\right\}.
\end{split}
\end{equation*}

Let $\{d^h\}$ be a approximating sequence in $W^{1,\infty}(\Omega, \mathbb{R}^3)$, satisfying:
\begin{equation}\label{dh}
d^h\to  \vec d \quad \mbox{ strongly in } L^2(\omega, \mathbb{R}^3), \quad \mbox{ and}\quad
h^2\|d^h\|_{W^{1, \infty}} \to 0.
\end{equation}
Note that such sequence can always be derived by reparametrizing
(slowing down) a sequence of smooth approximations of $ \vec d$.
Similiarly, consider the approximations $y^h\in W^{2,\infty}(\omega,
\mathbb{R}^3)$ and $ b^h\in W^{1,\infty}(\omega,\mathbb{R}^3)$,
with the following properties:
\begin{equation}\label{dwa2}
\begin{split}
& y^h\to y \quad \mbox{ strongly in } W^{2,2}(\omega,\mathbb{R}^3),
\quad \mbox{and } ~~  b^h\to  \vec b \quad \mbox{ strongly in }
W^{1,2}(\omega,\mathbb{R}^3)\\
& h\left( \|y^h\|_{W^{2,\infty}} + \|
  b^h\|_{W^{1,\infty}}\right)\leq \epsilon\\
&\frac{1}{h^2}|\omega\setminus \omega_h| \to 0, \quad \mbox{where }~~
\omega_h =\left\{x'\in\omega; ~ y^h(x') =  y(x') \mbox{ and }  b^h(x')
  =  \vec 
b(x')\right\} 
\end{split}
\end{equation}
for some appropriately small $\epsilon> 0$.
Existence of approximations with the claimed properties follows by
partition of unity and truncation arguments, as a special case of the
Lusin-type result for Sobolev functions (see {Proposition 2 in \cite{FJMhier})}.

We now define the recovery sequence $\xi^h\in W^{1,\infty}(\Omega^h, \mathbb{R}^3)$ by:
\begin{equation*}
\xi^h(x', x_3) = y^h(x') + x_3  b^h(x') + \frac{x_3^2}{2}d^h(x').
\end{equation*}
Consequently, the rescalings $y^h\in W^{1,\infty}(\Omega^1,
\mathbb{R}^3)$ are:
$$ y^h(x', x_3) = y^h(x') + hx_3  b^h(x') + \frac{h^2}{2}x_3^2 d^h(x'),$$
and so in view of (\ref{dh}) and (\ref{dwa2}), Theorem
\ref{thm3} (i) follows directly. The remaining convergence in (ii) is
achieved via standard calculations exactly as in \cite{BLS}. We suppress the details.
\end{proof}

It now immediately follows that:

\begin{corollary}\label{upper}
Existence of a $W^{2,2}$ regular immersion satisfying
(\ref{cinque}) is equivalent to the upper bound on the energy scaling at minimizers:
$$\exists C>0 \qquad \inf_{\xi\in W^{1,2}(\Omega^h,\mathbb{R}^3)} E^h(\xi) \leq Ch^2.$$
\end{corollary}

The following corollary is a standard conclusion of the
established $\Gamma$-convergence (see e.g. \cite{Braides}).  
It indicates that sequences of approximate solutions to the original problem on a
thin shell are in one-one correspondence to the minimizers of the thin
limit variational model  $\mathcal{I}_{G, \tilde G}$. 

\begin{corollary}\label{infmin}
Assume that (\ref{cinque}) admits a $W^{2,2}$-regular solution $y$.  Then
any sequence of approximate minimizers $\xi^h$ of (\ref{energyh}),  
satisfying the property:
$$ \lim_{h \to 0} \frac 1{h^2} \big(E^h(\xi^h)- \inf E^h\big) =0, $$ 
converges, after the proper rescaling and up to a subsequence (see
Theorem \ref{thm2}), to a minimizer of the functional $\mathcal{I}_{G, \tilde G}$. 
In particular, $\mathcal{I}_{G, \tilde G}$ attains its
minimum. Conversely, any minimizer $y$ of $\mathcal{I}_{G, \tilde G}$
is a limit  of approximate minimizers $\xi^h$ to (\ref{energyh}).    
\end{corollary}

On a final note, observe that the defect $\kappa(G, \tilde G)$, as
defined in \eqref{defect}), is of the order  
$h^2  \mathcal\min ~\mathcal{I}_{G, \tilde G} + o(h^2)$. It would be hence of
interest to discuss the necessary and sufficient conditions for having
$\min \mathcal{I}_{G, \tilde G}=0$, which in case of $G=\mbox{Id}_3$
have been precisely derived in \cite{BLS}. 

\section{An equivalent system of PDEs for \eqref{trzy}}\label{equivsec}
  
In this section, we investigate the integrability conditions of the
system (\ref{trzy}).
Firstly, recall that since the Levi-Civita connection is metric-compatible, we have:
\begin{equation}\label{chris}
\partial_i G_{jk} = G_{mk}\Gamma_{ij}^m + G_{mj}\Gamma_{ik}^m,
\end{equation}
where the Christoffel symbols (of second kind) of the metric $G$ are:
$$\Gamma^i_{kl} = \frac{1}{2} G^{im} (\partial_l G_{mk} + \partial_kG_{ml} - \partial_mG_{kl}).$$
Above, we used the Einstein summation
over the repeated upper and lower indices from $1$ to $n$.

\begin{lemma}\label{lem1}
Assume that there is a bilipschitz map $\xi:\Omega\to
U:=\xi(\Omega)$, between two open, bounded subsets: $\Omega, U$ 
of $\mathbb{R}^n$, satisfying (\ref{trzy}) and such that $\xi\in W^{2,2}(\Omega, U)$.

Then, denoting by $\tilde\Gamma_{ij}^k$ the Christoffel symbols of $\tilde
G$, and by $\Gamma_{ij}^k$ the Christoffel symbols of the metric
$G\circ \xi^{-1}$ on $U$, there holds, 
in $\Omega$:
\begin{equation}\label{cztery}
\forall i,j,s:1\ldots n \qquad 
\partial_{ij} \xi^s =\partial_m\xi^s\tilde\Gamma_{ij}^m
- \partial_i\xi^p\partial_j\xi^q(\Gamma_{pq}^s\circ \xi).
\end{equation} In particular $\xi$ automatically enjoys higher
regularity: $\xi\in W^{2,\infty}(\Omega, U)$. 
\end{lemma}
\begin{proof}
By (\ref{chris}) we obtain:
$$\partial_i \tilde G_{jk} = \big((\nabla\xi)^T G (\nabla
\xi)\big)_{mk}\tilde\Gamma_{ij}^m 
+ \big((\nabla\xi)^T G (\nabla \xi)\big)_{mj}\tilde\Gamma_{ik}^m,$$
while differentiating \eqref{trzy} directly, gives:
$$\partial_i \tilde G_{jk} = \big((\nabla\partial_i\xi)^T G (\nabla
\xi)\big)_{jk} + \partial_i\xi^p 
\big((\nabla\xi)^T \partial_p(G\circ\xi^{-1}) (\nabla
\xi)\big)_{jk} + \big((\nabla\xi)^T G
(\nabla\partial_i\xi)\big)_{jk},$$ where we used \cite[Theorem
2.2.2]{ziemer} to conclude that $G\circ \xi^{-1} \in W^{1,\infty}$ and
to apply the chain rule.  
Equating both sides above, and using \eqref{chris} to the
Lipschitz metric $G\circ\xi^{-1}$, we get:
\begin{equation*}
\begin{split}
\partial_m\xi^s G_{st}\partial_k\xi^t \tilde\Gamma^m_{ij} & + \partial_m\xi^s G_{st}\partial_j\xi^t \tilde\Gamma^m_{ik}
\\ & = \partial_{ij}\xi^s G_{st}\partial_k\xi^t + \partial_{j}\xi^s \partial_{i}\xi^p \partial_{k}\xi^t
\big(G_{tm}\Gamma_{ps}^m + G_{sm}\Gamma_{pt}^m\big) + \partial_{ik}\xi^s G_{st}\partial_j\xi^t, 
\end{split}
\end{equation*}
which we rewrite as:
\begin{equation}\label{sth}
\begin{split}
& \partial_m\xi^s \tilde\Gamma^m_{ij} \big(G\nabla\xi\big)_{sk}
+ \partial_m\xi^s \tilde\Gamma^m_{ik}\big(G\nabla\xi\big)_{sj}  
\\ & = \partial_{ij}\xi^s \big(G\nabla\xi\big)_{sk}
+ \partial_{j}\xi^q \partial_{i}\xi^p \Gamma_{pq}^s\big(G\nabla\xi\big)_{sk} + 
\partial_{j}\xi^p \partial_{k}\xi^t \Gamma_{pt}^s \big(G\nabla\xi\big)_{sj} 
+ \partial_{ik}\xi^s \big(G\nabla\xi\big)_{sj}. 
\end{split}
\end{equation}
For each $i,j:1\ldots n$ we now define the vector
$P_{ij}\in\mathbb{R}^n$ with components:
$$P_{ij}^s = \partial_{ij}\xi^s - \partial_m\xi^s\tilde\Gamma_{ij}^m
+ \partial_j\xi^q\partial_i\xi^p\Gamma_{pq}^s, \qquad s:1\ldots n. $$
By \eqref{sth}, we see that:
$$\forall i,j,k:1\ldots n\qquad \langle P_{ij},
\big(G\nabla\xi)_{k-col}\rangle = - \langle P_{ik}, \big(G\nabla\xi)_{j-col}\rangle, $$
where $\big(G\nabla\xi)_{k-col}$ is the vector denoting the $k$-th
column of the matrix $G\nabla\xi$.

Since $P_{ij} = P_{ji}$, it now follows that:
\begin{equation}\label{pij}
\forall i,j,k:1\ldots n\qquad \langle P_{ij},
\big(G\nabla\xi)_{k-col}\rangle = 0.
\end{equation}
But the columns of the invertible $G\nabla\xi$ form a linearly
independent system, so there must be $P_{ij}=0$, which completes the
proof of (\ref{cztery}).
\end{proof}

\begin{corollary}\label{boot}
Let $\xi:\Omega \to U$ be a bilipschitz map satisfying
(\ref{trzy}) as in Lemma \ref{lem1}. Assume that $\xi \in W^{2,2}(\Omega, U)$. 
Then $\xi$ and $\zeta= \xi^{-1}$
are both smooth and bounded, together with all their derivatives.
In particular, (\ref{cztery}) holds everywhere in $\Omega$.
\end{corollary}
\begin{proof}
It was already established that $\nabla^2 \xi \in L^\infty$. We have
$F:= \nabla \zeta = (\nabla \xi \circ \zeta)^{-1} \in L^\infty$.
Notice that $F^{-1}$ 
is the composition of two Lipschitz mappings and hence it is
Lipschitz. We conclude that for all $i$, $\partial_i F = -
F \partial_i (F^{-1}) F  \in L^\infty$, which implies that   
$\nabla \zeta \in W^{1,\infty}$, and hence $\zeta \in W^{2,\infty}(U, \Omega)$. 

By (\ref{trzy}) we get the following formula: 
$(\nabla\zeta)^T (\tilde G\circ\zeta) (\nabla\zeta) = G\circ\zeta$,
valid  in $U$. By the same calculations as in Lemma
\ref{lem1}, it results in:
\begin{equation*}
\forall i,j,s:1\ldots n \qquad 
\partial_{ij} \zeta^s =\partial_m\zeta^s\Gamma_{ij}^m - \partial_i\zeta^p\partial_j\zeta^q\tilde\Gamma_{pq}^s.
\end{equation*}
In view of (\ref{cztery}) and by a bootstrap argument, we obtain that
$\xi,\zeta\in W^{k,\infty}$ for every $k\geq 1$. Hence the result follows.
\end{proof}

\begin{corollary}\label{lem2}
Equivalently, \eqref{cztery} can be written as:
\begin{equation}\label{piec}
\begin{split}
\forall i,j,s:1\ldots n \quad 
\partial_{ij} \xi^s = ~& \partial_m\xi^s\tilde\Gamma_{ij}^m \\  &
-\frac{1}{2}G^{sm}  \Big(\partial_i\xi^q\partial_j G_{mq}
+ \partial_j\xi^p\partial_iG_{mp}
- \partial_j\xi^p\partial_i\xi^q\partial_tG_{pq} \big((\nabla\xi)^{-1}\big)_{tm}\Big).
\end{split}
\end{equation}
\end{corollary}
\begin{proof}
Denoting $\zeta=\xi^{-1}$ as before, we obtain:
$$\Gamma_{pq}^s = \frac{1}{2} G^{sm} \Big(\partial_p\zeta^t\partial_t G_{mq}
+ \partial_q\zeta^t\partial_tG_{mp} - \partial_m\zeta^t\partial_t
G_{pq} \Big).$$
Inserting in \eqref{cztery} and contracting
$\partial_p\zeta^t\partial_j\xi^p $ to the Kronecker delta $\delta_j^t$,
we obtain \eqref{piec}.
\end{proof}

\begin{theorem}\label{lem3}
Consider the following system of the algebraic-differential equations
in the unknowns $\xi, w_i:\Omega\to\mathbb{R}^n$, $i:1\ldots n$: 
\begin{subequations}\label{system} 
\begin{align}
& \tilde G_{ij} = w_i^t G_{st} w_j^s \label{i} \\
& w_i^t = \partial_i\xi^t \label{ii}\\
& \partial_{j} w_i^s =  w_m^s\tilde\Gamma_{ij}^m 
-\frac{1}{2}G^{sm}  \Big(w_i^p\partial_j G_{mp}
+ w_j^p\partial_iG_{mp}
- w_j^p w_i^q\partial_tG_{pq} W_m^t\Big), \label{iii}
\end{align}
\end{subequations}
where the matrix $[W_m^t]_{t,m=1..n}$ is defined as the inverse of the
matrix field $w:\Omega\to\mathbb{R}^{n\times n}$,
whose columns are the vectors $w_i$, i.e.: $W_m^t =
\big(\sum_{i=1}^n w_i\otimes e_i\big)^{tm}$.
Then we have the following:
\begin{itemize}
\item[(i)] Problem \eqref{trzy} has a solution given by a bilipschitz
  map $\xi \in W^{2,2}(\Omega, U)$ (as in Corollary \ref{boot}),  if and only if
  \eqref{system} has a solution $(\xi, w)$ given by a bilipschitz 
  $\xi:\Omega\to U$ and $w\in W^{1,2} (\Omega,\mathbb{R}^{n\times n})$.
  
\item[(ii)]  
Problem \eqref{system} has a solution, understood as in
(i) above, if and only if \eqref{iii} has a solution. This statement
should be understood in the following sense.

Assume that (\ref{iii}) is solved in the sense of distributions, by the vector fields $w_i\in
L^\infty(\Omega, \mathbb{R}^n)$, $i:1\ldots n$, such that $W_m^t$ are
well defined and $W_m^t\in L^\infty(\Omega)$. Then there exists a
smooth $\xi:\Omega\to U$ such that (\ref{ii}) holds. Moreover, $\xi$ is locally invertible to a smooth
vector field $\zeta$, and the Christoffel symbols of the following
metrics: 
$$\mathcal{G} =
G\circ\zeta \quad \mbox{ and } \quad \tilde{\mathcal{G}}=(\nabla\zeta)^T(\tilde
G\circ\zeta) (\nabla\zeta)$$ 
are the same. If additionally $\xi$ is globally invertible to
$\zeta:U\to\Omega$, and if we have:
\begin{equation}\label{init}
\exists x_0\in\Omega\qquad \tilde G(x_0) = w^T G w (x_0),
\end{equation}
then (\ref{i}) holds in $\Omega$.
\end{itemize}
\end{theorem}
\begin{proof}
Clearly, the equivalence in (i) follows from Lemma \ref{lem1} and
Lemma \ref{lem2}. For the equivalence in (ii) note first that,  
by a bootstrap argument, an $L^\infty$ vector field $w$ satisfying
(\ref{iii}) is automatically smooth in $\Omega$, together with $W$.
Further, in a simply connected domain
$\Omega$, the condition \eqref{ii} is the same as:
$$\forall i,j,t:1\ldots n\qquad \partial_j w_i^t = \partial_i w_j^t,$$
which  is implied by \eqref{iii}, by the symmetry of its right hand
side (in $i,j$). 

We now prove that the metrics ${\mathcal{G}}$ and
$\tilde{\mathcal{G}}$ have the same Christoffel symbols on the
subdomain of $U$ where the local inverse $\zeta$ is defined. Note that
both $\xi$ and $\zeta$ are smooth. We first compute:
\begin{equation}\label{sth1}
\partial_m\Big(\partial_i \zeta^p \tilde{\mathcal{G}}_{ps} \partial_j\zeta^s\Big) = 
\partial_{im} \zeta^p \tilde{\mathcal{G}}_{ps} \partial_j\zeta^s +
\partial_{i} \zeta^p \tilde{\mathcal{G}}_{ps} \partial_{jm}\zeta^s + 
\partial_{i} \zeta^p \partial_{j}\zeta^s \partial_{m} \zeta^q  \partial_q\tilde{\mathcal{G}}_{ps}. 
\end{equation}
By Lemma \ref{lem1}, we also obtain:
$$\partial_j w_i^s = w_m^s\tilde\Gamma_{ij}^m-w_j^p
w_i^q\Gamma^s_{pq},$$ 
where $\Gamma_{ij}^k$ are the Christoffel symbols of ${\mathcal{G}}$.
Now, since: $\partial_p\zeta^s w_i^p = \delta_i^s = \partial_i\zeta^p
w_p^s$, it follows that:
$$\partial_{pq}\zeta^s w_j^q w_i^p = -\partial_p\zeta^s \partial_j w_i^p 
= - \partial_p\zeta^s \Big(w_m^p\tilde\Gamma_{ij}^m - w_j^\alpha
w_i^\beta\Gamma^p_{\alpha\beta}\Big) = - \tilde\Gamma_{ij}^s + \partial_p\zeta^s w_j^\alpha
w_i^\beta\Gamma^p_{\alpha\beta}.$$
Consequently, \eqref{sth1} becomes:
\begin{equation}\label{sth2}
\begin{split}
 \partial_m\tilde{\mathcal{G}}_{ij} & = 
\partial_{im} \zeta^p \tilde{\mathcal{G}}_{ps} \partial_j\zeta^s +
\partial_{i} \zeta^p \tilde{\mathcal{G}}_{ps} \partial_{jm}\zeta^s \\ 
&\qquad \qquad\qquad\quad + 
\partial_{i} \zeta^p \partial_{j}\zeta^s \partial_{m} \zeta^q
\tilde{\mathcal{G}}_{\alpha p} \Big( -\partial_{\gamma
  \delta}\zeta^\alpha w_q^\gamma w_s^\delta +\partial_t\zeta^\alpha
w_q^\gamma w_s^\delta \Gamma_{\gamma\delta}^t\Big) \\ & 
\qquad\qquad\qquad\quad
+ \partial_{i} \zeta^p \partial_{j}\zeta^s \partial_{m} \zeta^q
\tilde{\mathcal{G}}_{\alpha s} \Big( -\partial_{\gamma
  \delta}\zeta^\alpha w_q^\gamma w_p^\delta +\partial_t\zeta^\alpha
w_q^\gamma w_p^\delta \Gamma_{\gamma\delta}^t\Big) \\ &
= \partial_i\zeta^p\partial_t\zeta^\alpha \tilde{\mathcal{G}}_{\alpha
  p}\Gamma_{mj}^t + \partial_j\zeta^s\partial_t\zeta^\alpha \tilde{\mathcal{G}}_{\alpha
  s}\Gamma_{im}^t \\ & = \tilde{\mathcal{G}}_{ti}\Gamma_{mj}^t +
\tilde{\mathcal{G}}_{tj}\Gamma_{mi}^t.
\end{split}
\end{equation}

Call $\gamma_{ij}^k$ the Christoffel symbols of $\tilde{\mathcal{G}}$,
so that: $\partial_m\tilde{\mathcal{G}}_{ij} = \tilde{\mathcal{G}}_{ti}\gamma_{mj}^t +
\tilde{\mathcal{G}}_{tj}\gamma_{mi}^t.$
Therefore:
$$\forall i,j,m:1\ldots n\qquad \tilde{\mathcal{G}}_{tj}(\Gamma_{mi}^t
- \gamma_{mi}^t) = - \tilde{\mathcal{G}}_{ti}(\Gamma_{mj}^t
- \gamma_{mj}^t).$$
Using the same reasoning as in the proof of \eqref{pij}
we get, as claimed:
$$\forall i,j,m:1\ldots n\qquad \Gamma_{ji}^m = \gamma_{ij}^m.$$
This concludes the proof of (ii) in view of Lemma \ref{lemchris2}
below, and since $\tilde{\mathcal{G}}(\xi(x_0)) = \mathcal{G}(\xi(x_0))$.
\end{proof}

\begin{lemma}\label{lemchris2}
Assume that the Christoffel symbols of two smooth metrics 
$\mathcal{G}_1, \mathcal{G}_2$ on a connected domain $U \subset
{\mathbb R}^n$ coincide, and that for some $x_0\in U$,  
$\mathcal{G}_1(x_0)=  \mathcal{G}_2(x_0)$. 
Then $\mathcal{G}_1=  \mathcal{G}_2$.
\end{lemma}
\begin{proof}  Let $\nabla_i$ represent the covariant derivative
  associated with the metric $\mathcal{G}_i$ through the Levi-Civita
  connection  \cite[p. 114 and Theorem 2.2, p. 158]{kono}. An
  immediate consequence of  
the equivalence of the Christoffel symbols (in the identity coordinate
system of $U$) is that $\nabla:= \nabla_1 = \nabla_2$
\cite[p. 144]{kono}. By \cite[Proposition 2.1, p. 158] {kono}, we have
$\nabla_i \mathcal{G}_i=0$,  i.e.   
both $\mathcal{G}_1$ and $\mathcal{G}_2$ are parallel tensor fields on
$U$ with respect to $\nabla$ \cite[ pages 88 and 124]{kono}. Let $x\in
U$ and let $\gamma$ be any piece-wise $\mathcal{C}^1$   
curve connecting $x_0$ and $x$ in $U$.  Let $\tau_\gamma$ be the
parallel transport from the point $u_0$ to $u$ along $\gamma$. It
follows that: 
$$ \tau_\gamma (\mathcal G_i(x_0)) = \mathcal{G}_i(x) \qquad \forall i=1,2,$$ 
which concludes the proof. 
\end{proof}

\begin{remark}\label{lemchris}
Assume that the Christoffel symbols of two smooth $2$d metrics 
$\mathcal{G}_1, \mathcal{G}_2$ on a connected domain
$U\subset\mathbb{R}^2$ coincide. Assume 
moreover that  the Gaussian curvature 
of $\mathcal{G}_1$ (and hence of $\mathcal{G}_2$) does not vanish
identically in $U$. Then, one can directly prove that there exists 
a positive constant $\lambda$ such that $\mathcal{G}_1= \lambda
\mathcal{G}_2$ in $U$.

Namely, let $\nabla_i$ represent the covariant derivative associated
with the metric $\mathcal{G}_i$ through the Levi-Civita connection, as
in the proof of Lemma \ref{lemchris2}. Again,
equivalence of the Christoffel symbols implies that $\nabla:= \nabla_1
= \nabla_2$.  Let  $x_0 \in U$ and let  
$H=\mbox{Hol}(U, x_0)$ be the holonomy group associated with $\nabla$
at $x_0$. By \cite{BorelLich}, (see also  \cite[Theorem 392]{berger}),  $H$ is a connected subgroup of the special  
orthogonal group associated with the scalar product
$\mathcal{G}_i(x_0)$. Since $(\Omega, \mathcal{G}_i)$ is not flat,
$H$ cannot be the trivial subgroup.  
The only other possibility (when $n=2$) is that $H$ is the entire
$SO(2, \mathcal {G}_i(x_0))$. This implies that the 
$SO(2, \mathcal {G}_1(x_0))= SO(2, \mathcal {G}_2(x_0))$, which by the
transitivity of the action of the special orthogonal group over the unit
sphere, results in $\mathcal{G}_1(x_0)= \lambda \mathcal{G}_2(x_0)$. 

Note that for $n>2$, the holonomy group $H$ generically coincides with
the full special orthogonal group \cite[p.643]{berger}. In this case again $SO(n, \mathcal
{G}_1(x_0))= SO(n, \mathcal {G}_2(x_0))$ for some $x_0\in U$, and so
the result is established for generic metrics $\mathcal{G}_i$. 
\end{remark}

 

\section{Some remarks and examples}

\begin{remark}
Note that in problem (\ref{trzy}) one can, without loss of generality,
assume that:
\begin{equation}\label{glom}
\tilde G(x_0) = G(x_0) = \mbox{Id}_{n}
\end{equation}
Indeed, denote $\tilde A= \tilde G(x_0)^{-1/2}$ and $A=
G(x_0)^{-1/2}$ the inverses of the unique symmetric square roots of
the metrics at $x_0$, and let $\Omega_1=\tilde A^{-1}(\Omega)$ be the
preimage of $\Omega$ under the linear transformation $x\mapsto \tilde
A x$. Then $\tilde G_1 = \tilde A (\tilde G\circ \tilde A)\tilde A$
and $G_2 = A(G\circ A)A$ are two metrics on $\Omega_1$, that equal
$\mbox{Id}$ at $x_1=\tilde A^{-1} x_0 $. Further, (\ref{trzy}) can be
written as:
\begin{equation}\label{trzyo}
\tilde G_1(x) = \Big(A^{-1}\big(\nabla\xi(\tilde Ax)\big) \tilde A\Big)^T
G_1(x) \Big(A^{-1}\big(\nabla\xi(\tilde Ax)\big) \tilde A\Big) =
\nabla \xi_1 G_1\nabla \xi_1(x),
\end{equation}
where $\xi_1 = A^{-1}\xi\circ \tilde A$. Clearly, existence of
a solution $\xi_1$ to (\ref{trzyo}) on $\Omega_1$ is completely
equivalent with existence of $\xi$ solving (\ref{trzy}) on $\Omega$,
with the same required regularity.

Notice also that, in view of the compensated regularity for $\xi$ in
Corollary \ref{boot}, any solution to (\ref{trzy}) satisfies, up to a
global reflection, the orientation preserving condition
(\ref{orienta}). In view of (\ref{glom}), (\ref{init}) now becomes: $w(x_0)\in SO(n)$.
For the particular case of $w(x_0) = \mbox{Id}$, it is easy to notice
that (\ref{iii}) at $x_0$ reduces to: $\partial_j w_i^s(x_0) =
\tilde\Gamma_{ij}^s - \chi_{ij}^s$, with $\chi_{ij}^s$ denoting the 
Christoffel symbols of the metric $G$ on $\Omega$.
\end{remark}

\begin{remark}\label{amitcalc}
We now present alternative calculations leading to the system
(\ref{i}) -- (\ref{iii}), adapting classical ideas
\cite{sokolnikoff}. These calculations do not require knowledge of the
metric compatibility of the connection.

\smallskip

{\bf 1.} Assume that there exists a bilipschitz map $\xi:\Omega\to
U$, whose global inverse we denote by $\zeta:U\to\Omega$, such that
(\ref{trzy}) holds in $\Omega$. Let $\tilde \Gamma_{ijk}$ the
Christoffel symbols of the first kind of the metric $\tilde G$ on
$\Omega$: 
$$\tilde \Gamma_{ijk} = \frac{1}{2}\big(\partial_i\tilde G_{jk}
+ \partial_j\tilde G_{ik} - \partial_k\tilde G_{ij}\big),$$
while let  $\Gamma_{\alpha\beta\gamma}$ stand for the Christoffel symbols of
the metric $\mathcal{G}=G\circ\zeta$ on $U$:
$$\Gamma_{\alpha\beta\gamma} = \frac{1}{2}\big(\partial_\alpha \mathcal{G}_{\beta\gamma}
+ \partial_\beta \mathcal{G}_{\alpha\gamma} - \partial_\gamma \mathcal{G}_{\alpha\beta}\big).$$
As above, we will use the Latin indices for components of vectors
in $\Omega$, and Greek indices in $U$.

Since by (\ref{trzy}) we have $\tilde G =
(\nabla\xi)^T(\mathcal{G}\circ \xi) (\nabla \xi)$, we obtain:
\begin{equation*}
\begin{split}
\partial_k\tilde G_{ij}
& = \partial_k\Big(\partial_i\xi^\alpha \partial_j\xi^\beta
(\mathcal{G}_{\alpha\beta}\circ \xi)\Big)  \\ & =
\Big( \partial_{ki}\xi^\alpha \partial_j\xi^\beta + \partial_{kj}\xi^\alpha\partial_i\xi^\beta \Big)
(\mathcal{G}_{\alpha\beta}\circ \xi) + \partial_i\xi^\alpha \partial_j\xi^\beta
\Big( (\partial_\gamma \mathcal{G}_{\alpha\beta})\circ \xi\Big) \partial_k\xi^\gamma.
\end{split}
\end{equation*}
Similarly, exchanging the following pairs of indices: $\gamma$ with
$\alpha$ and $\beta$ with $\gamma$ in the first equation, and $\gamma$
with $\beta$ in the second equation below, we get:
\begin{equation*}
\begin{split}
\partial_i\tilde G_{jk} &  = 
\Big( \partial_{ij}\xi^\alpha \partial_k\xi^\beta + \partial_{ik}\xi^\alpha\partial_j\xi^\beta \Big)
(\mathcal{G}_{\alpha\beta}\circ \xi) + \partial_i\xi^\alpha \partial_j\xi^\beta
\Big( (\partial_\alpha \mathcal{G}_{\gamma\beta})\circ
\xi\Big) \partial_k\xi^\gamma, \\
\partial_j\tilde G_{ik} &  = 
\Big( \partial_{ij}\xi^\alpha \partial_k\xi^\beta + \partial_{jk}\xi^\alpha\partial_i\xi^\beta \Big)
(\mathcal{G}_{\alpha\beta}\circ \xi) + \partial_i\xi^\alpha \partial_j\xi^\beta
\Big( (\partial_\beta \mathcal{G}_{\alpha\gamma})\circ
\xi\Big) \partial_k\xi^\gamma.
\end{split}
\end{equation*}
Consequently:
\begin{equation}\label{a1}
\tilde \Gamma_{ijk} = \partial_i\xi^\alpha \partial_j\xi^\beta \partial_k\xi^\gamma
\Big(\Gamma_{\alpha\beta\gamma} \circ \xi\Big)
+ \partial_{ij}\xi^\alpha \partial_k\xi^\beta \Big(\mathcal{G}_{\alpha\beta}\circ \xi\Big). 
\end{equation}

\smallskip

{\bf 2.} We now compute the Christoffel symbols of the second kind
for the metric $\tilde G$:
$$\tilde \Gamma_{ij}^k = \tilde G^{km}\tilde\Gamma_{ijm}.$$
Since $\tilde G^{-1} = \big((\nabla\zeta) \mathcal{G}^{-1}
(\nabla \zeta)^T\big)\circ\xi$, it follows that: $\tilde G^{km} =
(\partial_\alpha\zeta^k\partial_\beta\zeta^m\mathcal{G}^{\alpha\beta})\circ\xi$. By
(\ref{a1}) we get:
\begin{equation}\label{pomoc}
\begin{split}
\tilde\Gamma_{ij}^k &  = 
\Big((\partial_{\rho}\zeta^k \partial_\nu\zeta^m\mathcal{G}^{\rho\nu})\circ
\xi\Big) \partial_{i}\xi^\alpha\partial_j\xi^\beta \partial_m\xi^\gamma
(\Gamma_{\alpha\beta\gamma}\circ \xi)
+ \Big((\partial_{\rho}\zeta^k \partial_\nu\zeta^m\mathcal{G}^{\rho\nu})\circ
\xi\Big) (\mathcal{G}_{\alpha\beta}\circ\xi) \partial_{ij}\xi^\alpha \partial_m\xi^\beta
\\ &  = \Big((\mathcal{G}^{\rho\gamma}\Gamma_{\alpha\beta\gamma})\circ\xi\Big) 
\partial_{i}\xi^\alpha\partial_j\xi^\beta (\partial_\rho\zeta^k \circ
\xi) + \Big((\mathcal{G}^{\rho\beta}\mathcal{G}_{\alpha\beta})\circ\xi\Big) \partial_{ij}\xi^\alpha
(\partial_{\rho}\zeta^k\circ\xi)  \\ &
= (\Gamma_{\alpha\beta}^\rho\circ \xi) \partial_{i}\xi^\alpha\partial_j\xi^\beta (\partial_\rho\zeta^k \circ
\xi) + (\partial_{\alpha}\zeta^k\circ\xi) \partial_{ij}\xi^\alpha,
\end{split}
\end{equation}
where we contracted $\partial_m\xi^\gamma
(\partial_\nu\zeta^m\circ\xi)$ to the Kronecker delta
$\delta_\nu^\gamma$ and $\partial_m\xi^\beta (\partial_\nu\zeta^m\circ\xi)$ to
$\delta_\nu^\beta$ in the second equality, and
$\mathcal{G}^{\rho\beta}\mathcal{G}_{\alpha\beta}$ to
$\delta_\alpha^\rho$ in the third equality, where we also used the
definition of the Christoffel symbols $\Gamma_{\alpha\beta}^\rho$ of
the second kind of the metric $\mathcal{G}$. Further:
\begin{equation*}
\begin{split}
\tilde\Gamma_{ij}^k\partial_k\xi^\mu & =
(\partial_{\alpha}\zeta^k\circ\xi) \partial_k\xi^\mu \partial_{ij}\xi^\alpha
+ (\Gamma_{\alpha\beta}^\rho\circ \xi) \partial_k\xi^\mu \partial_{i}\xi^\alpha\partial_j\xi^\beta
(\partial_\rho\zeta^k \circ \xi)   \\ & =
\delta_\alpha^\mu \partial_{ij}\xi^\alpha + (\Gamma_{\alpha\beta}^\rho\circ \xi)
\delta_\rho^\mu \partial_{i}\xi^\alpha\partial_j\xi^\beta,
\end{split}
\end{equation*}
which implies the same formula as in (\ref{cztery}):
\begin{equation}\label{41}
\partial_{ij}\xi^\mu = \tilde\Gamma_{ij}^k\partial_k\xi^\mu - (\Gamma_{\alpha\beta}^\mu\circ \xi)
\partial_{i}\xi^\alpha\partial_j\xi^\beta,
\end{equation}

\smallskip

{\bf 3.} We now proceed as in Corollary \ref{lem2}:
\begin{equation*}
\begin{split}
\Gamma_{\alpha\beta}^\mu &  = 
\frac{1}{2}\mathcal{G}^{\mu\gamma}\Big(\partial_{\alpha} \mathcal{G}_{\beta\gamma} 
+ \partial_{\beta} \mathcal{G}_{\alpha\gamma} - \partial_{\gamma}
\mathcal{G}_{\alpha\beta} \Big) \\ & =
\frac{1}{2}\mathcal{G}^{\mu\gamma}\Big(
(\partial_sG_{\beta\gamma}\circ \zeta)\partial_\alpha\zeta^s + 
(\partial_rG_{\alpha\gamma}\circ \zeta)\partial_\beta\zeta^r -
(\partial_qG_{\alpha\beta}\circ \zeta)\partial_\gamma\zeta^q\Big).
\end{split}
\end{equation*}
Consequently, and in view of (\ref{41}):
\begin{equation*}
\begin{split}
\tilde \Gamma_{ij}^k\partial_k\xi^\mu - \partial_{ij}\xi^\mu & =
(\Gamma_{\alpha\beta}^\mu \circ\xi) \partial_i\xi^\alpha\partial_j\xi^\beta \\
& = \frac{1}{2}(\mathcal{G}^{\mu\gamma}\circ\xi) \Big(
\partial_sG_{\beta\gamma} (\partial_\alpha\zeta^s \circ\xi) \partial_i\xi^\alpha\partial_j\xi^\beta  + 
\partial_rG_{\alpha\gamma} (\partial_\beta\zeta^r
\circ\xi) \partial_i\xi^\alpha\partial_j\xi^\beta \\ &
\qquad\qquad\qquad\qquad\qquad \qquad\qquad\qquad  \quad -
\partial_qG_{\alpha\beta} (\partial_\gamma\zeta^q\circ
\xi) \partial_i\xi^\alpha\partial_j\xi^\beta \Big) \\ &
= \frac{1}{2} {G}^{\mu\gamma} \Big(
\partial_sG_{\beta\gamma} \delta_i^s \partial_j\xi^\beta  + 
\partial_rG_{\alpha\gamma} \delta_j^r \partial_i\xi^\alpha -
\partial_qG_{\alpha\beta} W_\gamma^q \partial_i\xi^\alpha\partial_j\xi^\beta \Big), 
\end{split}
\end{equation*}
which directly implies (\ref{iii}). \endproof
\end{remark}

\begin{remark}
We further observe that if (\ref{iii}) holds, then defining
$\xi$ by (\ref{ii}) we may again obtain (\ref{pomoc}) simply by
reversing steps in Remark \ref{amitcalc}. Letting $\bar G = 
(\nabla\xi)^T(\mathcal{G}\circ \xi) (\nabla \xi)$ and going through
the same calculations, the formula (\ref{pomoc}) follows,
this time for the Christoffel symbols $\bar \Gamma^k_{ij}$ of $\bar
G$, i.e.:
$$ \bar \Gamma_{ij}^k = (\Gamma_{\alpha\beta}^\rho\circ
\xi) \partial_{i}\xi^\alpha\partial_j\xi^\beta (\partial_\rho\zeta^k
\circ \xi) + (\partial_{\alpha}\zeta^k\circ\xi) \partial_{ij}\xi^\alpha. $$ 
Hence, we see that $ \tilde \Gamma_{ij}^k  = \bar \Gamma_{ij}^k$ in  $
\Omega$ and again, in view of Lemma \ref{lemchris2}, it follows that $\tilde G = \bar
G = (\nabla\xi)^T G (\nabla \xi)$ in  $\Omega$, provided that
we have this identity at a given point $x_0$, as required in (\ref{init}).
\end{remark}

\begin{example}\label{ex4.4}
Assume that $G$ is constant. Then the equations in (\ref{iii}) become:
\begin{equation}\label{iii1}
\partial_jw_i^s = w^s_m\tilde\Gamma^m_{ij},
\end{equation}
whereas the Thomas condition (\cite{thom}, see also next section) 
for the above system of differential equations is:
$$w_m^s\partial_k\tilde\Gamma^m_{ij} +
w^s_p\tilde\Gamma^p_{km}\tilde\Gamma^m_{ij} = w_m^s\partial_j\tilde\Gamma^m_{ik} +
w^s_p\tilde\Gamma^p_{jm}\tilde\Gamma^m_{ik} \qquad \forall
i,j,s,k:1\ldots n.$$
Equivalently, the following should hold: 
$$w_m^s\Big(\partial_k\tilde\Gamma^m_{ij} - \partial_j\tilde\Gamma^m_{ik} +
\tilde\Gamma^m_{kp}\tilde\Gamma^p_{ij} -
\tilde\Gamma^m_{jp}\tilde\Gamma^p_{ij} \Big) = 0 \qquad 
\forall i,j,s,k:1\ldots n,$$
at all points $x$ in a neighborhood of $x_0$, and for all $w$ in a
neighborhood of a given invertible
$w_0\in\mathbb{R}^{n\times n}$ which satisfies $\tilde G(x_0) = w_0^TG w_0$.
Since the expression in parentheses above equals $\tilde
R^m_{ijk}(x)$, it follows that the Thomas condition for (\ref{iii1})
is precisely the vanishing of the whole Riemann curvature tensor
$\tilde R^\cdot_{\cdot\cdot\cdot}$ of the metric $\tilde G$, or
equivalently that $\tilde G$ be immersible in $\mathbb{R}^n$.

On the other hand, letting $A=\sqrt{G}$ denote the unique positive
definite symmetric square root of the matrix $G$, we see that the
problem (\ref{trzy}) becomes:
$$\tilde G = (\nabla\xi)^T G \nabla \xi = (A\nabla\xi)^T
(A\nabla\xi) = (\nabla\xi_1)^T \nabla \xi_1,$$
with $\xi_1(x) = A\xi(x)$. Since $A$ is invertible, we easily deduce
that (\ref{trzy}) has a solution iff $\tilde G$ is immersible in $\mathbb{R}^n$. 
Consequently, in this example the Thomas condition for (\ref{iii}) is necessary and sufficient
for solvability of  (\ref{trzy}). 
\end{example}

\begin{example}\label{ex4.5}
Assume that $G(x) = \mu(x)\mbox{Id}_3$. Then the problem (\ref{trzy}) becomes:
\begin{equation}\label{trzy1} 
(\nabla \xi)^T\nabla \xi = \frac{1}{\mu}\tilde G = e^{2f}\tilde G,
\quad\mbox{ with } f=-\frac{1}{2}\log\mu.
\end{equation}
Solution to (\ref{trzy1}) exists if and only if the Ricci curvature of the metric
$e^{2f}\tilde G$, which is conformally equivalent to the metric $\tilde G$, is equal to
$0$ in $\Omega$. More precisely, denoting: $\nabla^2_{\tilde G} f = [\partial_{ij}f -
\tilde\Gamma_{ij}^k\partial_kf]_{i,j:1..3}\in\mathbb{R}^{3\times 3}$,
$\Delta_{\tilde G} f = \tilde G^{jk} \partial_{jk} f  - \tilde
G^{jk}\tilde\Gamma_{jk}^l\partial_lf\in\mathbb{R}$, and 
$\|\nabla f\|^2_{\tilde G} = \tilde G^{jk}\partial_j
f \partial_kf\in\mathbb{R}$, the condition reads:
\begin{equation}\label{trzy2} 
0 = \mbox{Ric}(e^{2f}\tilde G) = \mbox{Ric}(\tilde G) -
\Big(\nabla^2_{\tilde G} f - \nabla f\otimes \nabla f\Big) -
\Big(\Delta_{\tilde G} f + \|\nabla f\|^2_{\tilde G}\Big)\tilde G.
\end{equation}

When also $\tilde G = \lambda(x)\mbox{Id}_3$, then (\ref{trzy2}) after setting $h =
\frac{1}{2}\log\frac{\lambda}{\mu}$, reduces to:
\begin{equation}
\label{h-equ}  0 = - \big(\nabla^2 h - \nabla h\otimes \nabla h \big) -
\big(\Delta h + \|\nabla h\|^2\big)\mbox{Id}_3. 
\end{equation}
An immediate calculation shows that the only solutions of \eqref{h-equ} are:
$$ h(x)= -2 {\rm log}|x-a| + c \quad \mbox{ or } \quad  h= c, $$ 
with arbitrary constants $c\in \R$ and $a\in \R^n$. 
Indeed, in this case  the solutions to \eqref{trzy1} are conformal,
and so by Liouville's theorem they are given by the M\"obius
transformations in $\mathbb R^3$, as
compositions of rotations, dilations, inversions and translations of  the form:
$$ \xi(x)=  b +  \alpha \frac{ R(x- a)}{|x-a|^\beta},  \qquad R\in
SO(3), ~~a,b\in \R^3, ~~ \alpha\in \R, ~~\beta \in\{0, 2\}. $$
\end{example}

\begin{example}\label{lambda-mu}
In dimension $n=2$, existence of a solution to (\ref{trzy1})  is
equivalent to the vanishing of the Gauss curvature:
\begin{equation}\label{name}
0 = \kappa(e^{2f}\tilde G) = e^{-2f}\big(\kappa(\tilde G) -
\Delta_{\tilde G} f\big),
\end{equation}
which further becomes: $\kappa(\tilde G) = \Delta_{\tilde G}f$. 
When also $\tilde G = \lambda(x)\mbox{Id}_3$, this reduces to: $\Delta
h = 0$, which is also equivalent to the following compatibility of the
Gauss curvatures of $G$ and $\tilde G$:
$\frac{\kappa(\lambda\mbox{Id}_2)}{\kappa(\mu\mbox{Id}_2)} = \frac{\mu}{\lambda}.$
\end{example}


\section{Some further remarks on systems of total differential equations}\label{total_remarks}
 
Systems of total differential equations have been extensively studied
in the literature \cite{thom, kah, B, Eisenhart, veblen}. They are over-determined
systems  in the unknown $w:  \Omega\to \R^N$, of the form:
\begin{equation}\label{total-1}
\partial_i w (x)= f_i(x,w(x)), \qquad \forall i=1\ldots n\quad\forall x\in\Omega,
\end{equation} 
where all the partial derivatives of $w$ are given by functions 
$f_i:\Omega \times \R^N \to \R^N$. Note that if the latter are assumed of sufficient regularity, the uniqueness of solutions
to (\ref{total-1}) with a given initial data $w(x_0) = w_0$ is 
immediate. For existence, observe that 
any  solution must satisfy the compatibility conditions:
$$ \partial_{ij} w = \partial_{ji} w. $$
This leads, under sufficient regularity assumptions, to the necessary condition for existence of $w$:
\begin{equation}\label{comp1}
\Big(\frac{\partial f_i}{\partial x_j}  - \frac{\partial f_j}{\partial
    x_i} + \big(\frac{\partial f_i}{\partial w} \big)   f_j -
  \big(\frac{\partial f_j}{\partial w} \big) f_i \Big) (x, w(x))=0,
  \qquad \forall i,j:1\ldots n \quad \forall x\in\Omega.
\end{equation} 
The advantage of the system  \eqref{comp1} is that it does not
involve any partial derivatives of the a-priori unknown solution $w$, and hence, if certain
conditions are satisfied,  it can be used to obtain the candidates for $w$ by solving for 
$w(x)$ at each $x\in \Omega$. 

Naturally, the more solutions \eqref{comp1} has, the more there is a chance to
find a solution to \eqref{total-1}. A plausible strategy, 
possibly adaptable as practical numerical schemes,  
is to find all the candidates $w$ from \eqref{comp1} and check
whether they satisfy  \eqref{total-1}. This insight, combined with
the observation about the over-determination of the original  
system, implies its rigidity, and leads to non-existence of solutions in generic situations. 
On the other hand, the ideal situation 
is to have:
\begin{equation}\label{Thomas-1}
\frac{\partial f_i}{\partial x_j}  - \frac{\partial f_j}{\partial
    x_i} + \big(\frac{\partial f_i}{\partial w} \big)   f_j -
  \big(\frac{\partial f_j}{\partial w} \big) f_i \equiv 0,
  \qquad \forall i,j:1\ldots n.
\end{equation} 
satisfied for all $x\in \Omega$ and all $w\in \R^N$. All functions $w
: \Omega \to \R^N$ obtained this way are solutions to
 \eqref{comp1} and, as shown in \cite{thom}, 
this leads to existence of a family of local solutions to
\eqref{total-1}, parametrized by the initial values at a
given point $x_0\in \Omega$.  If \eqref{Thomas-1} is not satisfied,
one comes short of having an ample set of solutions $w$ and
might be content for other intermediate scenarios, where solutions to
\eqref{total-1} may still exist but with less liberty in
choosing initial values. 

In section \ref{equivsec} we showed that \eqref{trzy} is equivalent to
a system of total differential equations. Note that
even though our problem
shares some familiar features with the isometric immersion problem,
it is of a fundamentally different nature in as much as we cannot establish the equivalence of \eqref{total-1}  
and \eqref{Thomas-1}. Not being able to {close} the system
\eqref{comp1} as in the isometric immersion case, in order to find necessary and
sufficient conditions, we need to study the above mentioned possible  
intermediate scenarios when the Thomas condition \eqref{Thomas-1} is
not satisfied.  We will carry out this plan in this sections \ref{sec-duda} and \ref{suffisec}.
Below, we begin by a simple example of \eqref{total-1},  
whose conditions \eqref{Thomas-1} are far from optimal.

\begin{example}
 For $n=2, N=1$, consider the system: 
\begin{equation}\label{total} 
\nabla w= w^2 \vec a + w \vec b + \vec c \qquad \mbox{ in } \Omega\subset\mathbb{R}^2,
\end{equation}
where $\vec a, \vec b: \Omega \to \R^2$ are given
smooth vector fields and $\vec c \in \R^2$. In order to find the Thomas
condition \eqref{Thomas-1} for existence of a solution $w: \Omega \to \R$, we take:
$$ 0 = {\rm curl} ( w^2 \vec a + w \vec b+ \vec c) = w^2 {\rm curl}~
\vec a  +  2w \langle\nabla^\perp w, \vec a\rangle  + w {\rm curl} ~\vec b +
\langle\nabla^\perp w, \vec b  \rangle.$$
Substituting \eqref{total}, we obtain the counterpart of (\ref{comp1})
in the present case:
\begin{equation}\label{comp2} 
( {\rm curl}~  \vec a + \langle \vec a, \vec b^\perp\rangle) w^2 +
({\rm curl}~  \vec b) w = 0. 
\end{equation} 
The satisfaction of the above in the $(x, w)$ space is
precisely the condition \eqref{Thomas-1}:
\begin{equation}\label{Thomase}
{\rm curl} ~ \vec a + \langle a, \vec b^\perp \rangle\equiv 0  \quad
\mbox{and}\quad {\rm curl} ~\vec b\equiv 0 \quad \mbox{ in } \Omega. 
\end{equation}
\end{example}

\begin{lemma} \label{obs2}
Condition (\ref{Thomase}) is not necessary for the existence of solutions for (\ref{total}). 
Also, existence of a solution to (\ref{comp2}) does not
guarantee the existence of a solution to (\ref{total}). 
\end{lemma}
\begin{proof}
{\bf 1.} Let $w:\Omega \to \R$ be any positive smooth function. Let $\vec c
=0$ and let $\vec b : \Omega \to \R^2$ be a vector field which is not
curl free. Define:
$$ \vec a= \frac{1}{w^2} (\nabla w- w \vec b). $$
Then \eqref{total} is satisfied but not \eqref{Thomase}. 

\smallskip

{\bf 2.} Let $\vec a, \vec b: \Omega \to \R^2$ be two smooth vector
fields such that the first condition in (\ref{Thomase}) does not hold
at any point $x\in \Omega$. For example, one may take:
$$\vec a(x)= x \quad \mbox{ and } \quad \vec b(x)=  x^\perp,  $$
in any domain $\Omega$ which avoids $0$ in its closure. 
Then \eqref{comp2} becomes:
$$ |x|^2 w^2(x) + 2  w(x) =0, $$
and it has only two smooth solutions: $w \equiv 0$ and $w(x) = -2/|x|^2$. On the other
hand:
$$ \nabla (-\frac{2}{|x|^2})  = \frac{4}{|x|^4}x \neq - \big(\frac{2}{|x|^2}\big)^2 x - \frac{2}{|x|^2}
x^\perp + \vec c,$$
so none of these functions is a solution to (\ref{total}) when
$\vec c\neq 0$.
\end{proof}

Further, observe that augmenting \eqref{total} to a system of (decoupled) total differential equations:
\begin{equation}\label{totalsys}
\nabla w_I = w_I^2 \vec a_I + w_I \vec b_I + \vec c_I, \qquad I=1\ldots N,
\end{equation}
gives the condition:
\begin{equation}\label{pre-Thomas-sys} 
(  {\rm curl} ~ \vec a_I + \langle\vec a_I, \vec b_I^\perp\rangle  ) w_I^2
+  ({\rm curl} ~ \vec b_I) w_I =0 \qquad  \forall I=1\ldots N,
\end{equation} 
resulting in:
\begin{equation}\label{Thomas-sys} 
{\rm curl}  ~\vec a_I + \langle a_I, \vec b_I^\perp\rangle \equiv 0  \quad
\mbox{ and }\quad {\rm curl}~  \vec b_I \equiv 0,  \qquad \mbox{in }  \Omega. 
\end{equation}
As in Lemma \ref{obs2},  we can set up the
data such that, there exists a solution $ w_1(x)$ to the first
equation in the system  \eqref{pre-Thomas-sys} and that all the equations in
\eqref{Thomas-sys} for $I=2\ldots N$ are satisfied in such a manner
that \eqref{totalsys} still has no solutions.    

\section{The Thomas condition in the $2$-dimensional case of (\ref{trzy})}\label{sec-duda}

In this section, we assume that $n=2$ and follow the approach of
\cite{duda}. Recalling \eqref{equ-rotation} as the equivalent form of 
\eqref{trzy}, we note that it has a solution (on a simply connected $\Omega$) if and only if:
\begin{equation}\label{6.1}
 {\rm curl} \, (G^{-1/2} R \tilde G^{1/2}) \equiv 0 \quad \mbox{ in } \Omega,
\end{equation} 
for some rotation valued field $R :\Omega \to SO(2)$.
Denote $V= G^{1/2}$, $\tilde V= \tilde G^{1/2}$ and represent $R$ by a function
$\theta:\Omega \to \R$, so that $R(x)= R(\theta(x))$ is the rotation of angle
$\theta$. Note that:
$$ \partial_j R(\theta) = (\partial_j \theta)  R W,  \quad \mbox{
  where} \quad W=R(\frac{\pi}{2}) = \left[\begin{array}{cc}0&-1\\1&0\end{array}\right].$$ 
Since $W$ is a rotation, it commutes with all $R \in SO(2)$. Also:
$W^T = W^{-1} = -W$, and:
\begin{equation}\label{uno0} 
\forall F \in GL(2) \qquad WFW= -{\rm cof}\, F = -(\det F) F^{-1} .
\end{equation}

\smallskip

We finally need to recall the conformal--anticonformal decomposition of $2\times 2$
matrices. Let $\mathbb{R}^{2\times 2}_c$ and $\mathbb{R}^{2\times 2}_a$ denote,
respectively,  the spaces of conformal and anticonformal matrices:
$$\mathbb{R}^{2\times 2}_c =
\left\{\left[\begin{array}{cc} a& b\\ -b& a\end{array}\right]; ~ a,b\in\mathbb{R}\right\},
\qquad \mathbb{R}^{2\times 2}_a 
= \left\{\left[\begin{array}{cc} a& b\\ b& -a\end{array}\right]; ~ a,b\in\mathbb{R}\right\}.$$
It is easy to see that $\mathbb{R}^{2\times 2} =
\mathbb{R}^{2\times 2}_c \oplus \mathbb{R}^{2\times 2}_a$ because both
spaces have dimension $2$ and they are mutually orthogonal: $A:B = 0$ for all
$A\in  \mathbb{R}^{2\times 2}_c$ and $B\in \mathbb{R}^{2\times 2}_a$. 

For $F=[F_{ij}]_{i,j:1,2}\in\mathbb{R}^{2\times 2}$, its projections
$F^c $ on $\mathbb{R}^{2\times 2}_c$, and $F^a$ on $\mathbb{R}^{2\times 2}_a$ are:
$$F^c= \frac{1}{2}
\left[\begin{array}{cc} F_{11} + F_{22} & F_{12} - F_{21}\\ F_{21} -
    F_{12} & F_{11} + F_{22} \end{array}\right],
\qquad 
F^a= \frac{1}{2}
\left[\begin{array}{cc} F_{11} - F_{22} & F_{12} + F_{21}\\ F_{12} +
    F_{21} & F_{22} - F_{11} \end{array}\right].$$
We easily obtain that:
\begin{equation}\label{uno1} 
F= F^c + F^a, \quad |F|^2 = |F^c|^2 + |F^a|^2 \quad \mbox{ and } \quad \det F= 2( |F^c|^2 - |F^a|^2).
\end{equation}
Also:
\begin{equation}\label{uno2} 
\forall R \in SO(2) \qquad F^cR = RF^c  \quad \mbox{and} \quad R^TF^{a} = F^{a}R. 
\end{equation}

\begin{lemma}\label{lemduda}
The following system of total differential equations in $\theta$ is
equivalent with (\ref{trzy}):
\begin{equation}\label{equ-theta}
\nabla \theta =  \frac{1} {\det \tilde V} \tilde V  {\rm curl}\, \tilde V + 
 \frac{1} {\det \tilde V} \tilde V  \Big ( A_1 \tilde V   e_2 - A_2 \tilde V   e_1 \Big ) 
+  \frac{1} {\det \tilde V} \tilde VR(-2\theta)  \Big ( B_1 \tilde V   e_2 - B_2 \tilde V   e_1\Big ), 
\end{equation} 
where we have denoted:
\begin{equation}\label{AjBj} 
A_i= (V\partial_i  V^{-1})^c, \qquad  B_i= (V \partial_i  V^{-1})^{a}
\qquad \forall i=1,2. 
\end{equation} 
\end{lemma}
\begin{proof}
We calculate the expression in (\ref{6.1}): 
\begin{equation*}
\begin{split}
{\rm curl} (V^{-1} R \tilde V) & 
= \partial_1 (V^{-1} R \tilde V   e_2)- \partial_2 (V^{-1} R \tilde V   e_1)   \\ & =
(\partial_1 V^{-1} R \tilde V + V^{-1} R \partial_1 \tilde V )   e_2 
-  (\partial_2 V^{-1} R \tilde V + V^{-1} R \partial_2 \tilde V )  
e_1 \\ & \qquad + 
(\partial_1 \theta) V^{-1} RW \tilde V  e_2 - (\partial_2 \theta)V^{-1} RW \tilde V   e_1   
\end{split} 
\end{equation*}
The two last terms above equal, in view of (\ref{uno0}):
$$  (\partial_1 \theta) V^{-1} RW \tilde V W   e_1 + (\partial_2
\theta) V^{-1} RW \tilde V W e_2 = (V^{-1}RW\tilde V
W)(\nabla\theta)^T = (\det\tilde V) \big(V^{-1} R\tilde V^{-1}\big)(\nabla\theta)^T.$$ 
Therefore, by (\ref{6.1}) and (\ref{uno2}):
\begin{equation*}
\begin{split}
\nabla \theta &  = -\frac{1}{\det\tilde V} \big(\tilde V R^T \tilde V\big) \Big(-(\partial_1
V^{-1} R \tilde V + V^{-1} R \partial_1 \tilde V )   e_2  
+  (\partial_2 V^{-1} R \tilde V + V^{-1} R \partial_2 \tilde V )   e_1\Big) \\ & 
= \frac{1}{\det\tilde V} \Big(\tilde V R^T V \partial_1 V^{-1} R
\tilde V   e_2 - \tilde V^{-1} R^T V \partial_2 V^{-1} R \tilde V    e_1  
 + \tilde V^{-1} \partial_1 \tilde V    e_2 - \tilde V^{-1}  \partial_2 \tilde V   e_1\Big).
\end{split} 
\end{equation*}
We simplify the equation for $\theta$ as follows:
\begin{equation*}
\begin{split}
\nabla \theta & = \frac {1} {\det \tilde V}  \Big( \tilde V A_1 \tilde V   e_2 
-  \tilde V A_2 \tilde V  e_1 + \tilde V R(-2\theta)  B_1\tilde V   e_2 
 -   \tilde V R(-2\theta) B_2 \tilde V e_1 \\ 
 & \qquad\qquad\quad +  \tilde V  \partial_1 \tilde V    e_2 - \tilde V \partial_2 \tilde V   e_1\Big),
\end{split}
\end{equation*} 
which clearly implies (\ref{equ-theta}).
\end{proof}
 
\medskip

If the Thomas condition for (\ref{equ-theta}) is satisfied, there
exists a unique solution to (\ref{equ-theta}) for all initial values $\theta(x_0) =
\theta_0$. The original problem (\ref{trzy}) has then 
a unique solution for all initial values of the form 
$w(x_0)= \nabla \xi(x_0) = G^{-1/2} R_0 \tilde G^{1/2}(x_0)$ with $R_0 \in SO(2)$.

\begin{example}
Assume that $G={\rm Id}_2$. Then $A_i=B_i=0$, so that \eqref{equ-theta} becomes:
\begin{equation}\label{sterma}
\nabla \theta = h_{\tilde V}\quad\mbox{ where }  \quad h_{\tilde V} = \frac{1} {\det \tilde V} \tilde
V {\rm curl} \tilde V.  
\end{equation}
The condition for solvability of \eqref{sterma} is: ${\rm curl}
~h_{\tilde V} \equiv 0$, which is expected from Example
\ref{ex4.4}. Indeed, a direct calculation shows that (see also
\cite[Remark 2, page 113]{MD}) this condition is 
equivalent to the vanishing of the Gaussian curvature of $\tilde G$:
\begin{equation}\label{curve-form}  
\kappa (\tilde G) = - \frac{1}{\det \tilde V} {\rm curl} \, h_{\tilde V}.  
\end{equation} 
\end{example}  
 
\begin{example} 
As in Examples \ref{ex4.5} and \ref{lambda-mu}, assume that $ G =
\mu(x) {\rm Id}_2 = e^{-2f} {\rm Id}_2$, with $f$ given in (\ref{trzy1}). Then:
\begin{equation}\label{cosik} 
V \partial_i V^{-1} = -(\frac{1}{2\mu} \partial_i \mu) {\rm Id} =
(\partial_if) {\rm Id}_2. 
\end{equation}
This implies that $A_i  =  (\partial_i f)\, {\rm Id}_2$ and $B_i = 0$. Consequently, \eqref{equ-theta} becomes:
$$ \nabla \theta =  \frac{1} {\det \tilde V} \tilde V  {\rm curl}\, \tilde V 
 + \frac{1} {\det \tilde V} \tilde V^2  \big((\partial_1 f)e_2 -
 (\partial_2 f)e_1\big)  =   h_{\tilde V} + \frac{1} {\sqrt {\det \tilde G}} \tilde G ~ \nabla^\perp f. $$ 
But by (\ref{uno0}) it follows that:
\begin{equation*}
\begin{aligned} 
\frac{1} {\sqrt {\det \tilde G}}  \tilde G~  \nabla^\perp f & =  \frac{1} {\sqrt {\det \tilde G}}  \tilde G W \nabla f 
 =  \frac{1} {\sqrt {\det \tilde G}}  W^T W \tilde G W \nabla f =
 {\sqrt {\det \tilde G}} ~ W \tilde G^{-1} \nabla f   \\
 & =  \big(\sqrt {\det \tilde G}  ~\tilde G^{-1} \, \nabla f \big)^\perp,
\end{aligned}  
\end{equation*}
and hence:
$$ {\rm curl} \Big (\frac{1} {\sqrt {\det \tilde G}} ~\tilde G
~ \nabla^\perp f \Big ) = \mbox{div}\big(\sqrt{\det \tilde G} ~\tilde
G^{-1} \nabla f\big) = \sqrt {\det \tilde G} \, \Delta_{\tilde G} f. $$ 
Thus, the Thomas condition to (\ref{trzy}) is here:
$$ {\rm curl} \, h_{\tilde V} +\sqrt {\det \tilde G} \, \Delta_{\tilde
  G} f= 0, $$ 
and we see that,  in view of \eqref{curve-form}, it coincides with the equivalent
condition (\ref{name}) for existence of solutions to (\ref{trzy}).
\end{example}
   
\medskip

\begin{lemma}\label{lem_mn}
Without loss of generality and through a change of variable, we can assume that: 
$$\tilde G = \lambda(x) {\rm Id}_2 = e^{2g}  {\rm Id}_2\qquad
\mbox{with } \quad g= \frac{1}{2}\log\lambda.$$ 
Then, for an arbitrary metric $G$, condition (\ref{equ-theta}) becomes:  
\begin{equation}\label{equ-theta-m-n}
\nabla \theta =  \textswab{m}  +  R(-2\theta)  \textswab{n},
\end{equation} 
where $ \textswab{m} =  \nabla^\perp g +  (A_1e_2 - A_2
e_1)$ and $\textswab{n}=  B_1e_2 - B_2 e_1$.
The Thomas condition of (\ref{equ-theta-m-n}) reads:
\begin{equation}\label{Thomas-Theta}
\left \{
\begin{array}{l} 
{\rm curl} \,  \textswab{m} - 2 | \textswab{n}|^2 =0, \\
{\rm div}\,   \textswab{n} - 2 \langle\textswab{n}^\perp ,  \textswab{m}\rangle =0 , \\
{\rm curl}\,  \textswab{n} - 2 \langle  \textswab{n},  \textswab{m} \rangle=0.
\end{array}\right .
\end{equation} 
\end{lemma}
\begin{proof} 
As in (\ref{cosik}), we observe that $\tilde V\mbox{curl}\tilde V =
e^{2g}\nabla^\perp g$, and hence:
\begin{equation*} 
\nabla \theta = \nabla^\perp g +  \big(A_1 e_2  -  A_2 e_1 \big) 
+   R(-2\theta)  \big(B_1e_2 - B_2e_1\big).
\end{equation*} 
The above equation has a similar structure to (24) in
\cite{duda}, even though the two original problems are different. In
the present case, both $G$ and $\tilde G $ are involved in defining $ \textswab{m}$, while in
\cite{duda} the vector fields $ \textswab{m},\textswab{n}$ depend on
the matrix field $G$ in the Left Cauchy-Green equation: 
$(\nabla\eta) (\nabla \eta)^T =G$.

In order to derive the Thomas condition for (\ref{equ-theta-m-n}), note that:
\begin{equation*}
\begin{aligned}
{\rm curl} \nabla  \theta & = {\rm curl}\,  \textswab{m}
+ \partial_1 \langle R(-2\theta)  \textswab{n} ,   e_2\rangle  -  
\partial_2 \langle R(-2\theta)  \textswab{n},   e_1\rangle  \\ & =
{\rm curl} \,  \textswab{m} + \langle R(-2\theta) \partial_1  \textswab{n} ,   e_2 \rangle 
- \langle R(-2\theta) \partial_2  \textswab{n},   e_1 \rangle  -
2\partial_1 \theta \langle R(-2\theta) W  \textswab{n} ,  
  e_2 \rangle + 2 \partial_2\theta \langle R(-2\theta) W  \textswab{n},   e_1 \rangle.    
\end{aligned} 
\end{equation*} 
Substituting $\partial_i\theta$ from \eqref{equ-theta-m-n} we arrive at:
\begin{equation*}
\begin{aligned}
0 & = {\rm curl}\,  \textswab{m} + \langle R(-2\theta) \partial_1  \textswab{n} ,   e_2 \rangle 
- \langle R(-2\theta) \partial_2  \textswab{n},   e_1 \rangle \\ & 
\qquad -  2   \langle  \textswab{m} + R(-2\theta)  \textswab{n} ,   e_1 \rangle \langle R(-2\theta) W  \textswab{n} , 
W   e_1 \rangle  - 2  \langle   \textswab{m} + R(-2\theta)
\textswab{n} ,   e_2 \rangle \langle R(-2\theta) W  \textswab{n} ,  
W   e_2 \rangle \\ & =  {\rm curl} \,  \textswab{m} + \langle R(-2\theta) \partial_1  \textswab{n} ,   e_2 \rangle 
- \langle R(-2\theta) \partial_2  \textswab{n},   e_1 \rangle 
-  2   \langle  \textswab{m} + R(-2\theta)  \textswab{n} ,  R(-2\theta)  \textswab{n} \rangle \\ & =
 {\rm curl} \,  \textswab{m} + \langle    \partial_1  \textswab{n} ,  R(2\theta)  W   e_1 \rangle 
+ \langle  \partial_2  \textswab{n}, R(2\theta) W   e_2 \rangle - 2 |\textswab{n}|^2 
- 2\langle \textswab{m} ,  R(-2\theta)  \textswab{n} \rangle
\\ & =  {\rm curl} \,  \textswab{m} - 2 | \textswab{n}|^2 + \langle   \nabla  \textswab{n}\, :  R(2\theta)  W   \rangle 
  -2  \langle  \textswab{m} ,  R(-2\theta)  \textswab{n} \rangle \\ & 
  =  {\rm curl} \,  \textswab{m} - 2 | \textswab{n}|^2 + \langle   \nabla  \textswab{n} - 2
  W  \textswab{n} \otimes  \textswab{m} \,  :   R(2\theta) W  \rangle.  
 \end{aligned} 
\end{equation*}
Finally, writing $R(2\theta)W= R(2\theta+ \pi/2) = -\sin(2\theta) {\rm Id}_2 + \cos(2\theta)W$, we obtain:
\begin{equation*}
0 =  {\rm curl} \,  \textswab{m} - 2 | \textswab{n}|^2  - \sin
(2\theta) ({\rm div}\,   \textswab{n} - 2 \langle\textswab{n}^\perp ,  \textswab{m}\rangle) 
+ \cos (2\theta) ({\rm curl}\,  \textswab{n} - 2 \langle  \textswab{n},  \textswab{m} \rangle),  
\end{equation*} 
which should be satisfied for all $x\in\Omega$ and all $\theta\in[0,
2\pi)$, implying hence \eqref{equ-theta-m-n}.
\end{proof}

\begin{example} 
In the setting and using notation of Example \ref{lambda-mu}, we see that:
$ \textswab{m} = \nabla^\perp g + \nabla^\perp f = \nabla^\perp (f+g)
= \nabla^\perp h$ where $h=\frac{1}{2}\log\frac{\lambda}{\mu}$, and $\textswab{n} =0$.
The Thomas condition for (\ref{trzy}) here is hence: ${\rm curl} \,
\textswab{m} = 0$, that is equivalent to: 
$$ \Delta h = \Delta \log (\frac\lambda\mu) =0,$$
which is further exactly equivalent to existence of solutions in  Example \ref{lambda-mu}.
\end{example} 

\medskip

\begin{example} \label{coolexample}
We will now provide an example, where  the Thomas condition
\eqref{Thomas-Theta} is not  satisfied, but a solution to (\ref{trzy})
exists.
We start with requesting that $\tilde G = {\rm Id}_2$, which yields: $\textswab{m}=  A_1e_2 - A_2e_1$ and 
$\textswab{n} = B_1e_2 - B_2e_1$. Consider the general form of the diagonal
metric $G$:
$$ G(x)= \left [ \begin{array}{cc} e^{2a(x)} & 0 \\ 0 & e^{2b(x)} \end{array} \right],  $$ 
for some smooth functions $a,b:\bar\Omega\to\mathbb{R}$. Then:
$V \partial_iV^{-1}  = - \mbox{diag}\{\partial_ia, \partial_ib\}$, and hence:
$$ A_i = - \frac 12 \partial_i(a+b) \mbox{Id}_2,
\qquad B_i=\frac 12 \mbox{diag}\{\partial_i(b-a), \partial_i(a-b)\},$$
which leads to: 
$$  \textswab{m}= - \frac 12 \nabla^\perp (a+b) \quad \mbox{ and }
\quad \textswab{n} = \frac 12 \big(\partial_2(a-b), \partial_1 (a-b) \big).$$ 
Therefore, the Thomas condition \eqref{Thomas-Theta} becomes:
\begin{equation*} 
\left \{
\begin{array}{l} 
\Delta (a+b)  +  |\nabla (a-b)|^2 =0, \vspace{1mm} \\
\partial_{12}(a-b) + \frac{1}{2} \partial_1(a-b) \partial_2(a+b) 
+ \frac{1}{2}\partial_2(a-b) \partial_1(a+b) =0, \vspace{1mm}\\
  \partial_{11}(a-b) - \partial_{22}(a-b) + \partial_1(a-b) \partial_1(a+b) 
- \partial_2(a-b) \partial_2(a+b) = 0.
\end{array}\right .
\end{equation*}  
It can be simplified to the following form, symmetric in $a$ and $b$:
\begin{equation} \label{mstupid}
\left \{
\begin{array}{l} 
\Delta (a+b)  +  |\nabla (a-b)|^2 =0, \vspace{1mm} \\
2\partial_{12} a +  \partial_1a \partial_2a =  2\partial_{12} b + \partial_1b \partial_2b, \vspace{1mm}\\
  \partial_{11} a - \partial_{22} a + (\partial_1 a)^2 - (\partial_2 a)^2 =  
\partial_{11} b - \partial_{22} b + (\partial_1 b)^2 - (\partial_2 b)^2. 
\end{array}\right.
\end{equation} 

\smallskip

We now specify to the claimed example. Take $a=-\log x_1$ and $b=-\log x_2$, so that:
$$ G(x_1 ,x_2) = \left [ \begin{array}{cc} x_1 ^{-2} & 0 \\ 0 &
    x_2^{-2}  \end{array} \right ],  $$ 
defined on a domain $\Omega\subset \R^2$ in the positive quadrant, whose closure avoids $0$. 
We easily check that the first condition in  (\ref{mstupid}) does not hold, since:
$$ \Delta \log (x_1x_2)  +  |\nabla \log (\frac{x_1}{x_2})|^2 =
- \partial_{11} \log x_1 + |\partial_1 \log x_1|^2 - \partial_{22}
\log x_2 + |\partial_2 \log x_2|^2 = 2(\frac{1}{x_1^2} + \frac{1}{x_2^2}). $$ 
On the other hand,  $\xi_0(x_1,x_2)= \frac 12 (x_1^2, x_2 ^2)$ solves
(\ref{trzy}), because:
$$ (\nabla \xi_0)^T G\, \nabla \xi_0 = \left [ \begin{array}{cc} x_1 & 0
    \\ 0 & x_2  \end{array} \right ] \left [ \begin{array}{cc}
    x_1^{-2} & 0 \\ 0 & x_2^{-2}  \end{array} \right ] 
 \left [ \begin{array}{cc} x_1 & 0 \\ 0 & x_2  \end{array} \right ] =
 {\rm Id}_2 = \tilde G.$$ 
In fact, by reversing the calculations in the proof
of Lemma \ref{lem_mn}, in can be checked  that $\xi_0$ and $-\xi_0$
are the only two solutions to the problem (\ref{trzy}).
\end{example}

\section{A sufficient condition for the solvability of the problem
  (\ref{iii}) (\ref{init})}\label{suffisec}

In this section we go back to the general setting of $n\geq 2$
dimensional problem (\ref{trzy}) and its equivalent formulation (\ref{iii}).
Let $w_0\in\mathbb{R}^{n\times n}$ be an invertible matrix, such that:
$$\tilde G(x_0) = w_0^T G(x_0) w_0.$$ 
For every $s, i,j:1\ldots n$ let $f^{s,i}_j:\mathcal{O}\to\mathbb{R}$ be the smooth functions of $(x,
w)$, defined in a small neighborhood $\mathcal{O}$ of $(x_0, w_0)$, and
coinciding with the right hand side of (\ref{iii}):
$$f^{s,i}_j(x,w) =  w_m^s\tilde\Gamma_{ij}^m 
-\frac{1}{2}G^{sm}  \Big(w_i^p\partial_j G_{mp}
+ w_j^p\partial_iG_{mp} - w_j^p w_i^q\partial_tG_{pq} (w^{-1})_m^t\Big),$$
where $w_s^m$ denotes the element in the $s$-th row and $m$-th column
of the matrix $w$.

Differentiate $f_j^{s,i}$ formally in $\partial/\partial x_k$,
treating $w$ as a function of $x$, and substitute each partial
derivative $\partial_k w^p_q$ by the expression in $f_k^{p,q}$. We
call this new function $\tilde F^{s,i}_{j,k}$ and note that for $w$
satisfying (\ref{iii}) one has: $\partial_k\partial_j w_i^s(x) = \tilde
F^{s,i}_{j,k}(x, w(x))$ for every $x\in\Omega$.

Now, Thomas' sufficient condition \cite{thom} for the local
solvability of the total differential equation (\ref{iii}) with
the initial condition (\ref{init}) where $w(x_0)= w_0$, is that:
\begin{equation}\label{Thomas} 
\forall {s,i,k,j:1\ldots n} \qquad   F^{s,i}_{j,k} = \tilde F^{s,i}_{j,k} - \tilde F^{s,i}_{k,j} \equiv 0
\quad \mbox{in } \mathcal{O}. 
\end{equation}
Following \cite{LCG-acha}, we now derive another sufficient local
solvability condition. For convenience of the reader, we present the
whole argument as in section 4 \cite{LCG-acha}.

\medskip

For every $\alpha_1=1\ldots n$, let $F^{s,i}_{j,k,\alpha_1}=0$ be the
equation obtained after formally differentiating $F_{j,k}^{s,i}=0$ in
$\partial/\partial {x_{\alpha_1}}$ and replacing each partial derivative
$\partial_{\alpha_1}w_l^r$ by $f^{r,l}_{\alpha_1}$ as
before. Inductively, for every $m\geq 0$ and all $\alpha_1,
\alpha_2\ldots \alpha_m:1\ldots n$, we define the equations
$F^{s,i}_{j,k,\alpha_1\ldots \alpha_m}=0$.

\begin{theorem}\label{algebraic}
Fix $N\geq 0$. Assume that there exists a set $S\subset\{1\ldots
n\}^2$ of cardinality $1\leq \# S = M \leq n^2$, and there exist $M$ equations
$\bar F^1=0$, $\bar F^2=0$, \ldots $\bar F^M=0$, among the equations:
\begin{equation}\label{eqns}
\big\{F^{s,i}_{j,k} = 0\big\}_{s,i,j,k} \cup
\big\{F^{s,i}_{j,k,\alpha_1} = 0\big\}_{s,i,j,k,\alpha_1} \cup \ldots \cup
\big\{F^{s,i}_{j,k,\alpha_1\ldots \alpha_N} = 0\big\}_{s,i,j,k,
  \alpha_1\ldots \alpha_N},  
\end{equation}
such that:
\begin{equation}\label{invert}
\det\left[\frac{\partial \bar F^r}{\partial w_q^p} (x_0, w_0)
\right]_{\hspace{-2mm}\begin{array}{l} {\scriptstyle{r:1\ldots M}}
    \vspace{-2mm}\\  {\scriptstyle{(p,q)\in S}}\end{array}} \neq
0.\vspace{-1.5mm} 
\end{equation}
From now on, we will denote by $w^p_q$ the coefficients in the given
matrix $w$ with $(p,q)\in S$ and by $\bar w_t^l$ the coefficients with
$(l,t)\not\in S$. By the implicit function theorem, (\ref{invert})
guarantees existence of smooth functions $v^p_q:\mathcal{U}\times
\mathcal{V}\to\mathbb{R}$  for every $(p,q)\in S$, where $\{v^p_q(x, \{\bar w^l_t\}_{(l,t)\not\in
  S})\}_{(p,q)\in S}\in\mathbb{R}^M$ is defined for $x$ in a small
neighborhood $\mathcal{U}$ of $x_0$ and for $\{\bar v^l_t\}$ in a small
neighborhood $\mathcal{V}$ of $\{(\bar w_0)^l_t\}_{(l,t)\not\in S}$, satisfying:
\begin{equation}\label{gooda}
\begin{split}
& \forall (p,q)\in S \qquad v^p_q(x_0, \{(\bar w_0)^l_t\}) = (\bar w_0)^p_q,\\ 
& \forall r:1\ldots M \quad \forall x\in\mathcal{U} \quad \forall
\{\bar v^l_t\}\in \mathcal{V}\qquad \bar F^r\Big(x, \{\bar v^l_t\}_{(l,
  t)\not\in S}, \big\{v^p_q(x, \{\bar v^l_t\}_{(l, t)\not\in S})\big\}_{(p,
  q)\in S}\Big) =0.
\end{split}
\end{equation}
Assume further that:
\begin{equation}\label{satis}
\begin{split}
\forall m:0\ldots N+1 \quad \forall i,s,j,\alpha_1\ldots \alpha_m
\quad & \forall x\in\mathcal{U} \quad \forall
\{\bar v^l_t\}\in \mathcal{V}\qquad\\
& F^{s,i}_{j,k,\alpha_1\ldots\alpha_m}\Big(x, \{\bar v^l_t\}, \big\{v^p_q(x, \{\bar v^l_t\})\big\}\Big) =0.
\end{split}
\end{equation}
Then the problem (\ref{iii}) (\ref{init}) has a solution, defined in some small  
neighborhood $\mathcal{U}$ of $x_0$, such that $w(x_0) = w_0$.
\end{theorem}
\begin{proof}
{\bf 1.} Below, we drop the Einstein summation convention and use the
$\sum$ sign instead, for more clarity. By $\mathcal{U}$ and
$\mathcal{V}$ we always denote appropriately small neighborhoods of
$x_0\in\mathbb{R}^n$ and $\{(\bar w_0)^l_t\}_{(l,t)\not\in
  S}\in\mathbb{R}^{n^2 - M}$, respectively, although the sets $\mathcal{U}$ and
$\mathcal{V}$ may vary from step to step. 

We will seek for a solution in the form $\Big(\{\bar
w^l_t(x)\}, \{w^p_q(x) = v^p_q\big(x, \{\bar w^l_t(x)\}\big)\}\Big)$, where the
functions $v^p_q$ are defined in the statement of the theorem.
Note first that by (\ref{satis}), for every $\alpha:1\ldots n$, every
$m:0\ldots N$, and $s,i,j,k,\alpha_1\ldots\alpha_m$,
after denoting: $F = F^{s,i}_{j,k,\alpha_1\ldots \alpha_m}$ we have:
\begin{equation}\label{13}
\frac{\partial F}{\partial x_\alpha} (x,v)  + \sum_{(l,t)\not\in S}
\frac{\partial F}{\partial \bar w^l_t}(x,v) f^{l,t}_\alpha(x, v)  +  \sum_{(p,q)\in S}
\frac{\partial F}{\partial  w^p_q}(x,v) f^{p,q}_\alpha(x, v) = 0, 
\end{equation}
for all $x\in\mathcal{U}$ and all $\{\bar v^l_t\}_{(l,t)\not\in
  S}\in\mathcal{V}$, where above $v = \Big\{\{\bar
v^l_t\}_{(l,t)\not\in S}, \{v^p_q\}_{(p,q)\in S}\Big\}$ and $v^p_q(x)
  = v^p_q\big(x, \{\bar v^l_t(x)\}\big)$ for all $(p,q)\in  S$.

\medskip

{\bf 2.} Fix now any point $x\in\mathcal{U}$ and let
$[0,\eta]\ni\tau\mapsto g(\tau) = \{g^\alpha(\tau)\}_{\alpha=1\ldots n}\in
\mathcal{U}$ be a smooth path with constant speed, such that:
$g(0)=x$. Consider the solution $[0,\eta]\ni\tau\mapsto h(\tau) =
\{h^l_t(\tau)\}_{(l,t)\not\in S}\in\mathcal{V}$ of the
following initial value problem:
\begin{equation}\label{ode}
\left\{ \begin{aligned}
& \frac{\mbox{d}}{\mbox{d}\tau} h^l_t(\tau) = \sum_{\alpha=1\ldots n}
f^{l,t}_\alpha\Big( g(\tau), h(\tau), \{v^p_q(g(\tau), h(\tau))\}\Big)
\frac{\mbox{d}g^\alpha}{\mbox{d}\tau}(\tau)\qquad \forall \tau\in[0,1]
\quad\forall (l,t)\not\in S\\
& h^l_t(0) = \bar v^l_t.
\end{aligned} \right.
\end{equation}
Applying (\ref{satis}), we easily see that for $F$ denoting, as in
step 1, any function in the set of equation (\ref{eqns}), there holds:
$$ F\Big(g(\tau), h(\tau), \big\{v^p_q(g(\tau), h(\tau))\big\}\Big) = 0 \qquad
\forall \tau\in [0,\eta].$$
Differentiating in $\tau$ yields:
\begin{equation*}
\begin{split}
& \sum_{\alpha=1\ldots n} \frac{\partial F}{\partial x_\alpha}(g,h, \{v^p_q(g,h)\})
\frac{\mbox{d}g^\alpha}{\mbox{d}\tau} +
\sum_{(l,t)\not\in S} \frac{\partial F}{\partial \bar w^l_t}(g,h, \{v^p_q(g,h)\})
\frac{\mbox{d}h^l_t}{\mbox{d}\tau} \\ & 
\qquad\qquad +
\sum_{(p,q)\in S} \frac{\partial F}{\partial w^p_q}(g, h, \{v^p_q(g,h)\})
\Big(\frac{\partial v^p_q}{\partial x_\alpha}(g, h) \frac{\mbox{d}g^\alpha}{\mbox{d}\tau} +
\frac{\partial v^p_q}{\partial \bar v^l_t}(g, h)
\frac{\mbox{d}h^l_t}{\mbox{d}\tau}\Big) = 0.
\end{split}
\end{equation*}
We now use (\ref{13}) to equate the first term in the expression
above. In view of (\ref{ode}), the second term cancels out and we
obtain, for the particular choice of $F=\bar F^r$, $r:1\ldots M$:
\begin{equation*}
\begin{split}
& \forall \tau\in [0,\eta] \quad \forall r:1\ldots M \qquad \\ &
\sum_{\alpha=1\ldots n}  \frac{\mbox{d}g^\alpha}{\mbox{d}\tau}(\tau) 
\sum_{(p,q)\in S} \frac{\partial \bar F^r}{\partial w^p_q}(g, h, \{v^p_q(g,h)\})
\Bigg(\sum_{(l,t)\not\in S} \frac{\partial v^p_q}{\partial \bar
  v^l_t}(g, h) f^{l,t}_\alpha(g,h, \{v^p_q(g,h)\}) \\ &
\qquad\qquad\qquad\qquad \qquad\qquad\qquad\qquad  \qquad\quad + 
\frac{\partial v^p_q}{\partial x_\alpha}(g, h) - f^{p,q}_\alpha(g,h,
\{v^p_q(g,h)\})\Bigg) = 0.
\end{split}
\end{equation*}
We proceed by evaluating the obtained formula at $\tau=0$. Since
${\mbox{d}g}/{\mbox{d}\tau}\neq 0$ is an arbitrary vector in $\mathbb{R}^n$, and since the
matrix $\left[\frac{\partial \bar F^r}{\partial w^p_q}(x, \{\bar v^l_t\})\right]$
is invertible by (\ref{invert}), it follows that:
\begin{equation}\label{formu}
\begin{split}
\forall \alpha:1\ldots n \quad \forall (p,q)\in S \qquad &
\Bigg(\sum_{(l,t)\not\in S} \frac{\partial v^p_q}{\partial \bar
  v^l_t}(x, \{\bar v^l_t\}) f^{l,t}_\alpha(x, \{\bar v^l_t\},
\{v^p_q(x, \{\bar v^l_t\})) \Bigg) \\ & \qquad\qquad + 
\frac{\partial v^p_q}{\partial x_\alpha}(x, \{\bar v^l_t\}) - f^{p,q}_\alpha(x, \{\bar v^l_t\},
\{v^p_q(x, \{\bar v^l_t\})\}) = 0,
\end{split}
\end{equation}
for all $x\in\mathcal{U}$ and $\{\bar v^l_t\}\in\mathcal{V}$.

\medskip

{\bf 3.} Consider the following system of total differential equations:
\begin{equation}\label{auxi}
\forall \alpha:1\ldots n \quad \forall (l,t)\not\in S \qquad 
\frac{\partial \bar w^l_t}{\partial x_\alpha} = f^{l,t}_\alpha\Big(
x, \{\bar w^l_t\}_{(l,t)\not\in S}, \{v^p_q(x, \big\{\bar
w^l_t\big\})\}_{(p,q)\in S} \Big).
\end{equation}
We now verify the Thomas condition for the system (\ref{auxi}). It
requires \cite{thom} the following expressions to be symmetric in $j,k:1\ldots n$, for all
$(i,s)\not\in S$ on $\mathcal{U}\times \mathcal{V}$: 
\begin{equation*}
\begin{split}
& \frac{\partial f^{i,s}_k}{\partial x_j}(x, \{\bar w^l_t\},
\{v^p_q\}) + \sum_{(l,t)\not\in S} \frac{\partial f^{i,s}_k}{\partial \bar
  w^l_t}(x, \{\bar w^l_t\},  \{v^p_q\}) f^{l,t}_j\big(x, \{\bar w^l_t\},
\{v^p_q\}\big) \\ & \qquad\qquad + 
\sum_{(p,q\in S)}\frac{\partial f^{i,s}_k}{\partial w^p_q}(x, \{\bar w^l_t\},  \{v^p_q\})\Bigg(
\frac{\partial v^p_q}{\partial x_j}(x, \{\bar w^l_t\}) +
\sum_{(l,t)\not\in S} \frac{\partial v^p_q}{\partial \bar v^l_t}(x,
\{\bar w^l_t\}) f^{l,t}_k \big(x, \{\bar w^l_t\},  \{v^p_q\}\big)\Bigg).
\end{split}
\end{equation*}
By (\ref{formu}), the above expression equals:
\begin{equation*}
\begin{split}
& \frac{\partial f^{i,s}_k}{\partial x_j}(x, \{\bar w^l_t\},
\{v^p_q\}) + \sum_{(l,t)\not\in S} \frac{\partial f^{i,s}_k}{\partial \bar
  w^l_t}(x, \{\bar w^l_t\},  \{v^p_q\}) f^{l,t}_j(x, \{\bar w^l_t\},
\{v^p_q\}) \\ & \qquad\qquad \qquad\qquad\quad
+ \sum_{(p,q)\in S}\frac{\partial f^{i,s}_k}{\partial w^p_q}(x, \{\bar
w^l_t\},  \{v^p_q\})  f^{p,q}_j (x, \{\bar w^l_t\},  \{v^p_q\}),
\end{split}
\end{equation*}
where as usual $v^p_q = v^p_q(x, \bar w^l_t)$.  We see that the required
symmetry follows exactly by (\ref{13}).

Consequently, the problem (\ref{auxi}) with the initial condition  $\bar w^l_t(x_0) = (w_0)^l_t$ has a unique
solution $\{\bar w^l_t\}_{(l,t)\not\in S}$ on $\mathcal{U}$.

\medskip

{\bf 4.} Let now $\{w^p_q\}_{(p,q)\in S}$ be defined on $\mathcal{U}$
through the formula: $w^p_q(x) = v^p_q(x, \{\bar w^l_t(x)\}_{(l,
  t)\not\in S})$. We will prove that:
\begin{equation}\label{auxi2}
\forall \alpha:1\ldots n \quad \forall (p,q)\in S \qquad 
\frac{\partial  w^p_q}{\partial x_\alpha} = f^{p,q}_\alpha\Big(
x, \{\bar w^l_t\}_{(l,t)\not\in S}, \{w^p_q\}_{(p,q)\in S} \Big).
\end{equation}
Together with (\ref{auxi}), this will establish the result claimed in the theorem.

Differentiating (\ref{gooda}) where we set $v=w$, and applying (\ref{auxi}) yields:
\begin{equation*}
\begin{split}
& \forall r:1\ldots M \quad \forall \alpha:1\ldots n \quad \forall
x\in\mathcal{U} \qquad \\ &
\frac{\partial\bar F^r}{\partial x_\alpha}(x, w(x)) + 
\sum_{(l,t)\not\in S} \frac{\partial \bar F^r}{\partial \bar
  w^l_t}(x, w(x)) f^{l,t}_\alpha(x, w(x))  + \sum_{(p,q)\in S}
\frac{\partial \bar F^r}{\partial w^p_q}(x, w(x)) \frac{\partial
  w^p_q}{\partial x_\alpha}(x) = 0.
\end{split}
\end{equation*}
In view of (\ref{13}), applied to $F=\bar F^r$, $r=1\ldots M$
and $v=w$, the desired equality in (\ref{auxi2}) follows directly by
the invertibility of the matrix
$\left[\frac{\partial \bar F^r}{\partial w^p_q}(x,
  w)\right]_{\hspace{-2mm}\begin{array}{l} {\scriptstyle{r:1\ldots M}} 
    \vspace{-2mm}\\  {\scriptstyle{(p,q)\in S}}\end{array}}$, 
which by (\ref{invert}) is valid in a sufficiently small neighborhood
$\mathcal{O}$ of $(x_0, w_0)$.
\end{proof}

\begin{remark}
Another similar approach \cite{LCG-acha},  would be as
follows. In case the Thomas condition \eqref{Thomas} is not
satisfied, one could relax the initial condition \eqref{init} and
instead add the set of equations \eqref{i} to those in \eqref{Thomas}, before proceeding to derive the 
collections of equations in (\ref{eqns}) by successive differentiation and
substitution, as described above. Assume then that
\eqref{invert} holds in a large domain rather than a given point
$(x_0, w_0)$.  The advantage of this method is that
one does not have to be concerned about satisfying the initial condition
potentially limiting the choices of the functions $\bar F^r$ in
(\ref{invert}), and hence one has a larger set of equations among
\eqref{eqns} to choose from.  Through this method,
existence of an $n^2 -M$ parameter family of
solutions to (\ref{iii}) follows, as one has the freedom to set up the initial data in step
3. The disadvantage is that a larger set  of equations in \eqref{satis}
must be satisfied for the sufficient condition to hold true.
\end{remark}

\begin{remark}

While Theorem \ref{algebraic} does provide exact algebraic conditions
of integrability for (\ref{trzy}), these conditions are by no means
easy to verify for specific problems. With a view towards a more
practical, if only approximate, solution procedure, the
following alternative may be considered. In the spirit of Section \ref{var_reform}, that
sets up an infinite-dimensional optimization problem whose zero-cost
solutions correspond to exact solutions of (\ref{trzy}) (and whose
non-zero cost solutions may be considered as approximate solutions of
(\ref{trzy})), we now construct a family of finite-dimensional
optimization problems parametrized by $x \in \Omega$ (uncoupled from
one $x$ to another), whose pointwise zero-cost solutions (for each
$x$) form the ingredients of an exact solution of (\ref{trzy}),
as outlined in Section \ref{total_remarks}. As before, we use the
Einstein summation convention for all lowercase Latin
indices from $1$ to $n$. 

Let $w \in \mathbb{R}^{n \times n},$ $x \in \Omega \subset \mathbb{R}^n$ and define:
\begin{equation}\nonumber
\begin{split}
& A_{ij} (w,x) = \tilde{G}_{ij}(x) - w^t_i G_{st}(x) w^s_j\\
& \tilde{B}^s_{ij}(w,x) = w^s_m \tilde{\Gamma}^m_{ij} (x) -
\frac{1}{2} G^{sm} (x) \left( w^p_i \partial_j G_{mp} (x) +
  w^p_j \partial_i G_{mp}(x) - w^p_j w^q_i \partial_t G_{pq} (x) W^t_m (w) \right)\\ 
& C^s_{ijk} (w,x) = \frac{\partial \tilde{B}^s_{ij}}{\partial x^k}
(w,x) - \frac{\partial \tilde{B}^s_{ik}}{\partial x^j} (w,x) +
\frac{\partial \tilde{B}^s_{ij}}{\partial w^r_p} (w,x)
\tilde{B}^r_{pk}(w,x) - \frac{\partial \tilde{B}^s_{ik}}{\partial
  w^r_p} (w,x) \tilde{B}^r_{pj}(w,x). 
\end{split}
\end{equation}
The symmetry of $\tilde{B}^s_{ij}$ in the lower indices and $C = 0$
represent the integrability conditions of (\ref{ii}) and (\ref{iii}), while
$A = 0$ is equivalent to (\ref{i}). Define:
\begin{equation}\label{algeb_cost}
\mathcal{E} (w;x) = K_1 \sum_{i,j=1}^n |A_{ij} (w,x)|^2 + K_2 \sum_{i,j,k,s=1}^n |C^s_{ijk} (w,x)|^2,
\end{equation}
with $K_i > 0, i = 1,2$ as nondimensional constants, where we assume
that $A,C$ in (\ref{algeb_cost}) have been appropriately
nondimensionalized.  

We now seek minimizers $w(x)$ of $\mathcal{E} (\cdot \ ;x)$ for each
$x \in \Omega$. Each set $\{w(x);~ x \in \Omega\}$ such that
$\mathcal{E} (w(x);x) = 0$ on $\Omega$, defines a function $x \mapsto
w(x)$, $x\in \Omega$. Each such function can be checked to see if it
satisfies (\ref{iii}), and any that passes this test allows for the
construction of a solution to (\ref{ii}) (because of the symmetry in
$i,j$) that in turn is a solution of (\ref{trzy}) in view of $A=0$.

If the system (\ref{ii}) and (\ref{iii}) were completely integrable
(i.e. satisfying the Thomas condition) then there would be a $(n^2 +
n)$-parameter family of solutions to (\ref{trzy}). While the situation
here is likely to be more constrained (with non-existence expected  in
the generic case), it is not possible to rule out, \emph{a-priori},
the existence of solution families with fewer parameters. Thus, it is
natural to expect nonuniqueness in looking for minimizers of
$\mathcal{E}(\cdot\ ; x)$. This idea has the merit of lending itself naturally to a computational algorithm.
\end{remark}

\end{document}